\documentclass[11pt,a4paper]{amsart}

\usepackage{amsmath}
\usepackage{mathtools}
\usepackage{amsthm}
\usepackage{amssymb}
\usepackage{hyperref}

\newcommand*\fullref[3][\relax]{%
  \ifdefined\hyperref%
    {\hyperref[#3]{#2\penalty 200\ \ref*{#3}#1}}%
  \else%
    {#2\penalty 200\ \relax\ref{#3}#1}%
  \fi%
}

\usepackage{tikz}
\usetikzlibrary{calc}
\usetikzlibrary{decorations.pathmorphing}
\usetikzlibrary{fadings}
\usetikzlibrary{shapes.misc}

\usepackage{etoolbox}

\usetikzlibrary{arrows.meta}

\tikzset{
  normalarrow/.style={line width=.6pt},
}

\tikzset{
  normalarrowlabel/.style={
    auto,
    inner sep=.5mm,
    outer sep=0mm,
    font=\footnotesize,
  },
  tinyarrowlabel/.style={
    auto
    inner sep=.2mm,
    outer sep=0mm,
    font=\tiny,
  },
  arrowlabel/.style={
    normalarrowlabel
  },
  marrowlabel/.style={
    normalarrowlabel,
  },
  crystaledges/.style={
    f1/.style={
      ->,
      normalarrow,
      labelled/.style={every to/.style={edge node={node[marrowlabel] {$1$}}}},
      coloured/.style={draw=red},
    },
    f2/.style={
      ->,
      normalarrow,
      labelled/.style={every to/.style={edge node={node[marrowlabel] {$2$}}}},
      coloured/.style={draw=blue},
    },
    f3/.style={
      ->,
      normalarrow,
      labelled/.style={every to/.style={edge node={node[marrowlabel] {$3$}}}},
      coloured/.style={draw=green!50!black},
    },
    f4/.style={
      ->,
      normalarrow,
      labelled/.style={every to/.style={edge node={node[marrowlabel] {$4$}}}},
      coloured/.style={draw=brown},
    },
    f5/.style={
      ->,
      normalarrow,
      labelled/.style={every to/.style={edge node={node[arrowlabel] {$5$}}}},
      coloured/.style={draw=cyan},
    },
    f6/.style={
      ->,
      normalarrow,
      labelled/.style={every to/.style={edge node={node[arrowlabel] {$6$}}}},
      coloured/.style={draw=violet},
    },
    f7/.style={
      ->,
      normalarrow,
      labelled/.style={every to/.style={edge node={node[arrowlabel] {$7$}}}},
      coloured/.style={draw=brown!50!black},
    },
    f8/.style={
      ->,
      normalarrow,
      labelled/.style={every to/.style={edge node={node[arrowlabel] {$7$}}}},
      coloured/.style={draw=magenta},
    },
    f9/.style={
      ->,
      normalarrow,
      labelled/.style={every to/.style={edge node={node[arrowlabel] {$7$}}}},
      coloured/.style={draw=brown!50!black},
    },
    f10/.style={
      ->,
      normalarrow,
      labelled/.style={every to/.style={edge node={node[arrowlabel] {$7$}}}},
      coloured/.style={draw=lime},
    },
    f11/.style={
      ->,
      normalarrow,
      labelled/.style={every to/.style={edge node={node[arrowlabel] {$7$}}}},
      coloured/.style={draw=olive},
    },
    f12/.style={
      ->,
      normalarrow,
      labelled/.style={every to/.style={edge node={node[arrowlabel] {$7$}}}},
      coloured/.style={draw=pink},
    },
    fi/.style={
      ->,
      normalarrow,
      labelled/.style={every to/.style={edge node={node[marrowlabel] {$i$}}}},
      coloured/.style={draw=gray},
    },
    fn2/.style={
      ->,
      normalarrow,
      labelled/.style={every to/.style={edge node={node[marrowlabel] {$n{-}2$}}}},
      coloured/.style={draw=brown},
    },
    fn1/.style={
      ->,
      normalarrow,
      labelled/.style={every to/.style={edge node={node[arrowlabel] {$n{-}1$}}}},
      coloured/.style={draw=cyan},
    },
    fn/.style={
      ->,
      normalarrow,
      labelled/.style={every to/.style={edge node={node[arrowlabel] {$n$}}}},
      coloured/.style={draw=violet},
    },
    df1/.style={
      f1,
      densely dotted,
    },
    df2/.style={
      f2,
      densely dotted,
    },
    df3/.style={
      f3,
      densely dotted,
    },
    dfi/.style={
      fi,
      densely dotted,
    },
    dfn2/.style={
      fn2,
      densely dotted,
    },
    dfn1/.style={
      fn1,
      densely dotted,
    },
    dfn/.style={
      fn,
      densely dotted,
    },
  },
  labelledcrystaledges/.style={
    crystaledges,
    f1/.append style={labelled},
    f2/.append style={labelled},
    f3/.append style={labelled},
    f4/.append style={labelled},
    f5/.append style={labelled},
    f6/.append style={labelled},
    f7/.append style={labelled},
    f8/.append style={labelled},
    f9/.append style={labelled},
    f10/.append style={labelled},
    f11/.append style={labelled},
    f12/.append style={labelled},
    fi/.append style={labelled},
    fn2/.append style={labelled},
    fn1/.append style={labelled},
    fn/.append style={labelled},
  },
  colouredcrystaledges/.style={
    crystaledges,
    f1/.append style={coloured},
    f2/.append style={coloured},
    f3/.append style={coloured},
    f4/.append style={coloured},
    f5/.append style={coloured},
    f6/.append style={coloured},
    f7/.append style={coloured},
    f8/.append style={coloured},
    f9/.append style={coloured},
    f10/.append style={coloured},
    f11/.append style={coloured},
    f12/.append style={coloured},
    fi/.append style={coloured},
    fn2/.append style={coloured},
    fn1/.append style={coloured},
    fn/.append style={coloured},
  },
  labelledcolouredcrystaledges/.style={
    crystaledges,
    f1/.append style={labelled,coloured},
    f2/.append style={labelled,coloured},
    f3/.append style={labelled,coloured},
    f4/.append style={labelled,coloured},
    f5/.append style={labelled,coloured},
    f6/.append style={labelled,coloured},
    f7/.append style={labelled,coloured},
    f8/.append style={labelled,coloured},
    f9/.append style={labelled,coloured},
    f10/.append style={labelled,coloured},
    f11/.append style={labelled,coloured},
    f12/.append style={labelled,coloured},
    fi/.append style={labelled,coloured},
    fn2/.append style={labelled,coloured},
    fn1/.append style={labelled,coloured},
    fn/.append style={labelled,coloured},
  },
  crystalvertex/.style={
    font=\small,
    inner sep=.5mm,
    outer sep=0mm,
  },
  smallcrystalvertex/.style={
    crystalvertex,
    font=\scriptsize,
  },
  bigcrystalvertex/.style={
    crystalvertex,
    inner sep=1mm,
    font=\normalsize,
  },
  crystal/.style={
    x=10mm,
    y=10mm,
    every node/.style={crystalvertex},
    labelledcrystaledges,
  },
  bigcrystal/.style={
    x=15mm,
    y=15mm,
    every node/.style={bigcrystalvertex},
    labelledcrystaledges,
  },
  smallcrystal/.style={
    x=7mm,
    y=7mm,
    every node/.style={smallcrystalvertex},
    colouredcrystaledges,
  },
}

\usetikzlibrary{shapes.geometric} 
\usetikzlibrary{shapes.misc}  

\tikzset{
  bst/.style={
    standard/.style={
      font=\small,
      draw=gray,
      rounded rectangle,
      minimum width=4.5mm,
      minimum height=4.5mm,
      inner xsep=0mm,
      inner ysep=1mm,
      outer sep=0mm,
      line width=.5pt,
    },
    empty/.style={
      minimum width=3mm,
      minimum height=3mm,
    },
    triangle/.style={
      isosceles triangle,
      isosceles triangle apex angle=60,
      shape border rotate=90,
      rounded corners=2mm,
      minimum width=8mm,
      inner xsep=0mm,
      inner ysep=.5mm
    },
    blank/.style={
      draw=none,
    },
    nodecount/.style={
      blank,
      font=\scriptsize,
    },
    every node/.style={standard},
    every child/.style={draw=black,line width=.6pt},
    level distance=10mm,
    level 1/.style={sibling distance=60mm},
    level 2/.style={sibling distance=30mm},
    level 3/.style={sibling distance=15mm},
  },
  medbst/.style={
    bst,
    level distance=10mm,
    level 1/.style={sibling distance=15mm},
    level 2/.style={sibling distance=15mm},
    level 3/.style={sibling distance=15mm},
  },
  smallbst/.style={
    bst,
    level distance=8mm,
    level 1/.style={sibling distance=10mm},
    level 2/.style={sibling distance=10mm},
    level 3/.style={sibling distance=10mm},
  },
  tinybst/.style={
    bst,
    level distance=5mm,
    level 1/.style={sibling distance=8mm},
    level 2/.style={sibling distance=8mm},
    level 3/.style={sibling distance=8mm},
    every node/.append style={
      font=\footnotesize,
    },
    triangle/.append style={
      rounded corners=1mm,
      minimum width=7mm,
      inner xsep=-.5mm,
    },
  },
  microbst/.style={
    bst,
    standard/.append style={
      font=\scriptsize,
      minimum width=3mm,
      minimum height=3mm,
      inner ysep=.25mm,
    },
    level distance=3mm,
    level 1/.style={sibling distance=6mm},
    level 2/.style={sibling distance=6mm},
    level 3/.style={sibling distance=6mm},
  },
  nanobst/.style={
    bst,
    standard/.append style={
      font=\tiny,
      minimum width=2mm,
      minimum height=2mm,
      inner ysep=.25mm,
    },
    level distance=2mm,
    level 1/.style={sibling distance=4mm},
    level 2/.style={sibling distance=4mm},
    level 3/.style={sibling distance=4mm},
  },
}
\tikzfading[name=fade down,bottom color=transparent!0,top color=transparent!100]
\tikzfading[name=fade right,right color=transparent!0,left color=transparent!100]

\tikzset{
  mogrifyarrow/.style={
    ->,
    >/.tip=Computer Modern Rightarrow,
    decorate,
    decoration={
      zigzag,
      amplitude=0.2em,
      segment length=0.35em,
      pre length=0.35em,
      post length=0.35em,
    },
  },
}

\theoremstyle{definition}
\newtheorem{definition}{Definition}[section]

\newtheorem{algorithm}[definition]{Algorithm}

\newtheorem{method}[definition]{Method}

\newtheorem{remark}[definition]{Remark}

\theoremstyle{plain}
\newtheorem{corollary}[definition]{Corollary}
\newtheorem{lemma}[definition]{Lemma}
\newtheorem{proposition}[definition]{Proposition}
\newtheorem{theorem}[definition]{Theorem}

\numberwithin{equation}{section}
\newcommand*{\textparens}[1]{\textup{(}#1\textup{)}}

\newcommand*{\defterm}[1]{\emph{#1}}

\makeatletter
\newcommand\chyph{\penalty\@M-\hskip\z@skip}
\makeatother

\DeclarePairedDelimiter{\abs}{\lvert}{\rvert}

\DeclarePairedDelimiter{\parens}{\lparen}{\rparen}

\DeclarePairedDelimiter{\set}{\{}{\}}
\DeclarePairedDelimiterX{\gset}[2]{\{}{\}}{\,#1:#2\,}
\newcommand\gsetsplit[3][]{\mathopen#1\{\,#2:#3\,\mathclose#1\}}

\newcommand*{\biggg}{\bBigg@{4}}

\newcommand*{\Biggg}{\bBigg@{5}}

\makeatletter
\newcommand*{\sizeddelimiter}[2]{\bBigg@{#1}#2}
\makeatother

\newcommand*{\nset}{\mathbb{N}}

\newcommand*{\emptyword}{\varepsilon}

\newcommand*{\wlen}[2][]{\abs[#1]{#2}}

\DeclarePairedDelimiterX{\pres}[2]{\langle}{\rangle}{#1\,\delimsize\vert\,\mathopen{}#2}

\newcommand*{\drel}[1]{\mathcal{#1}}

\newcommand*\e{\ddot{e}}
\newcommand*\f{\ddot{f}}

\newcommand*\ecount{\ddot\epsilon}

\newcommand*\ke{\tilde{e}}
\newcommand*\kf{\tilde{f}}

\newcommand*\kecount{\tilde\epsilon}
\newcommand*\kfcount{\tilde\phi}
\newcommand*{\aA}{\mathcal{A}}
\newcommand*{\stdlit}{\mathrm{std}}
\newcommand*{\std}[2][]{\stdlit\parens[#1]{#2}}

\newcommand*{\clen}[2][]{\ell\parens[#1]{#2}}
\newcommand*{\cwt}[2][]{\abs[#1]{#2}}

\newcommand*{\wt}[2][]{\wtlit\parens[#1]{#2}}
\newcommand*{\wtlit}{\mathrm{wt}}

\newcommand*{\tlen}[2][]{\abs[#1]{#2}}

\newcommand*{\plac}{{\mathsf{plac}}}
\newcommand*{\hypo}{{\mathsf{hypo}}}
\newcommand*{\sylv}{{\mathsf{sylv}}}

\newcommand*{\baxt}{{\mathsf{baxt}}}

\newcommand*{\placcong}{\equiv_\plac}
\newcommand*{\hypocong}{\equiv_\hypo}

\newcommand*{\sylvcong}{\equiv_\sylv}

\newcommand*{\baxtcong}{\equiv_\baxt}

\newcommand*{\colreading}[2][]{\mathrm{C}\parens[#1]{#2}}

\newcommand*{\descomp}[2][]{\mathcal{DC}\parens[#1]{#2}}

\newcommand*{\cnp}[2][]{\mathrm{cnp}\parens[#1]{#2}}
\newcommand*{\dectree}[2][]{\mathrm{DecT}\parens[#1]{#2}}
\newcommand*{\inctree}[2][]{\mathrm{IncT}\parens[#1]{#2}}
\newcommand*{\ltree}[2][]{\mathrm{T}_{\mathrm{L}}\parens[#1]{#2}}
\newcommand*{\rtree}[2][]{\mathrm{T}_{\mathrm{R}}\parens[#1]{#2}}
\newcommand*{\lrectree}[2][]{\mathrm{TRec}_{\mathrm{L}}\parens[#1]{#2}}
\newcommand*{\rrectree}[2][]{\mathrm{TRec}_{\mathrm{R}}\parens[#1]{#2}}
\newcommand*{\inreading}[2][]{\mathrm{In}\parens[#1]{#2}}
\newcommand*{\postreading}[2][]{\mathrm{Post}\parens[#1]{#2}}

\newcommand*{\plit}{\mathrm{P}}
\newcommand*{\qlit}{\mathrm{Q}}

\newcommand*{\pplac}[2][]{\plit_{\plac}\parens[#1]{#2}}
\newcommand*{\qplac}[2][]{\qlit_{\plac}\parens[#1]{#2}}

\newcommand*{\phypo}[2][]{\plit_{\hypo}\parens[#1]{#2}}
\newcommand*{\qhypo}[2][]{\qlit_{\hypo}\parens[#1]{#2}}

\newcommand*{\psylv}[2][]{\plit_{\sylv}\parens[#1]{#2}}
\newcommand*{\qsylv}[2][]{\qlit_{\sylv}\parens[#1]{#2}}

\newcommand*{\pbaxt}[2][]{\plit_{\baxt}\parens[#1]{#2}}
\newcommand*{\qbaxt}[2][]{\qlit_{\baxt}\parens[#1]{#2}}

\newcommand*{\shape}[2][]{\shapelit\parens[#1]{#2}}
\newcommand*{\shapelit}{\sigma}

\newcommand*{\sylvshape}[2][]{\sylvshapelit\parens[#1]{#2}}
\newcommand*{\sylvshapelit}{\shapelit_{\sylv}}
\newcommand*{\sylvisom}{\sim_{\sylv}}

\newcommand*{\baxtshape}[2][]{\baxtshapelit\parens[#1]{#2}}
\newcommand*{\baxtshapelit}{\shapelit_{\baxt}}
\newcommand*{\baxtisom}{\sim_{\baxt}}

\newcommand*{\Sh}[2][]{\Shlit\parens[#1]{#2}}
\newcommand*{\Shlit}{\mathrm{Sh}}

\newcommand*{\placisom}{\approx}
\newcommand*{\hypoisom}{\sim}
\newcommand*{\shapisom}{\sim_{\shapelit}}

\begin{document}

\title[Crystals and trees]{Crystals and trees: quasi-Kashiwara operators, monoids of binary trees, and
  {R}obinson--{S}chensted-type correspondences}

\author{Alan J. Cain}
\address{%
Centro de Matem\'{a}tica e Aplica\c{c}\~{o}es\\
Faculdade de Ci\^{e}ncias e Tecnologia\\
Universidade Nova de Lisboa\\
2829--516 Caparica\\
Portugal
}
\email{%
a.cain@fct.unl.pt
}
\thanks{The first author was supported by an Investigador {\sc FCT} fellowship ({\sc IF}/01622/2013/{\sc CP}1161/{\sc
    CT}0001).}

\author{Ant\'onio Malheiro}
\address{%
Centro de Matem\'{a}tica e Aplica\c{c}\~{o}es and Departamento de Matem\'{a}tica\\
Faculdade de Ci\^{e}ncias e Tecnologia\\
Universidade Nova de Lisboa\\
2829--516 Caparica\\
Portugal
}
\email{%
ajm@fct.unl.pt
}
\thanks{For both authors, this work was partially supported by by the Funda\c{c}\~{a}o para
a Ci\^{e}ncia e a Tecnologia (Portuguese Foundation for Science and Technology) through the project {\sc UID}/{\sc
  MAT}/00297/2013 (Centro de Matem\'{a}tica e Aplica\c{c}\~{o}es), and the project {\scshape PTDC}/{\scshape MHC-FIL}/2583/2014.}

\thanks{The authors thank the anonymous referee for carefully reading the paper and for valuable comments, and, in
  particular, for making the observation contained in \fullref{Remark}{rem:slncrystals}.}

\begin{abstract}
  Kashiwara's crystal graphs have a natural monoid structure that arises by identifying words labelling vertices that
  appear in the same position of isomorphic components. The celebrated plactic monoid (the monoid of Young tableaux),
  arises in this way from the crystal graph for the $q$-analogue of the general linear Lie algebra $\mathfrak{gl}_{n}$, and
  the so-called Kashiwara operators interact beautifully with the combinatorics of Young tableaux and with the
  Robinson--Schensted--Knuth correspondence. The authors previously constructed an analogous `quasi-crystal' structure
  for the related hypoplactic monoid (the monoid of quasi-ribbon tableaux), which has similarly neat combinatorial
  properties. This paper constructs an analogous `crystal-type' structure for the sylvester and Baxter monoids (the
  monoids of binary search trees and pairs of twin binary search trees, respectively). Both monoids are shown to arise
  from this structure just as the plactic monoid does from the usual crystal graph. The interaction of the structure
  with the sylvester and Baxter versions of the Robinson--Schensted--Knuth correspondence is studied. The structure is
  then applied to prove results on the number of factorizations of elements of these monoids, and to prove that both
  monoids satisfy non-trivial identities.
\end{abstract}

\maketitle

\section{Introduction}

Kashiwara's notion of a crystal basis \cite{kashiwara_crystalizing,kashiwara_oncrystalbases} (informally, a basis of a
representation for a suitable algebra on which the generators have a particularly straightforward action), gives rise,
via tensor products, to the crystal graph. This graph has a natural monoid structure where the vertices are viewed as
words in the free monoid, and words are identified when they lie in the same places in isomorphic components of this
graph. In the case of the $q$-analogue of the general linear Lie algebra $\mathfrak{gl}_n$, the monoid obtained is the
celebrated plactic monoid, whose elements can be identified with Young tableaux and which appears in such diverse
contexts as symmetric functions \cite{macdonald_symmetric}, representation theory \cite{green_polynomial}, algebraic
combinatorics \cite{lothaire_algebraic}, algebraic geometry \cite{fulton_young}, Kostka--Foulkes polynomials
\cite{lascoux_plaxique,lascoux_foulkes}, Schubert polynomials \cite{lascoux_schubert,lascoux_tableaux}, and musical
theory \cite{jedrzejewski_plactic}. (Indeed, since Lascoux and Sch\"utzenberger's \cite{lascoux_plaxique} seminal study,
so many connections have emerged that Sch\"utzen\-berger proclaimed it `one of the most fundamental monoids in algebra'
\cite{schutzenberger_pour}.) The beautiful interaction of the crystal graph and the associated Kashiwara operators with
the combinatorics of Young tableaux and the Robinson--Schensted correspondence is such an important and powerful
combinatorial tool that in this context the Kashiwara operators and crystal graph are sometimes simply referred to as
the `coplactic' operators and `coplactic' graph (see, for instance, \cite[\S~5.5]{lothaire_algebraic} and
\cite{vanleeuwen_littlewood}).

In a previous paper \cite{cm_hypoplactic}, the present authors constructed an analogue of this crystal structure for the
monoid of quasi-ribbon tableaux: the so-called hypoplactic monoid. The hypoplactic monoid emerged from the theory of
non-commutative symmetric functions and quasi-symmetric functions
\cite{krob_noncommutative4,krob_noncommutative5,novelli_hypoplactic}. Quasi-ribbon functions form a basis for the ring
of quasi-symmetric functions, just as the Schur polynomials form a basis for the ring of symmetric polynomials. The
quasi-ribbon functions are indexed by the so-called quasi-ribbon tableaux, which can be identified with the elements of
the hypoplactic monoid. As shown in the authors' previous paper, much of the elegant interaction of the crystal graph,
Kashiwara operators, Young tableaux, and the plactic monoid is echoed in the interaction of the analoguous quasi-crystal
graph, quasi-Kashiwara operators, quasi-ribbon tableaux, and the hypoplactic monoid.

Thus both plactic and hypoplactic monoids can be defined by factoring the free monoid by a relation that can be defined
in three equivalent ways:
\begin{itemize}
\item[1.] \textit{Defining relations}: the relation is the smallest congruence containing a given set of defining relations.
\item[2.] \textit{Insertion}: the relation is the kernel of the map that takes a word to a `tableau-type' object computed by
   an insertion algorithm.
\item[3.] \textit{Crystals}: the relation arises from isomophisms of connected components of a `crystal-type' graph.
\end{itemize}

The main aim of this paper is to develop an analoguous of the crystal-type structure for the sylvester
\cite{hivert_sylvester} and Baxter \cite{giraudo_baxter2} monoids, whose elements are respectively right strict binary
search trees and pairs of twin binary search trees. To be more precise: the existing literature contains equivalent
definitions for these monoid using defining relations and insertions algorithms (approaches~1 and~2 above). This paper
develops and applies the crystalline definition (approach~3 above).

In fact, the crystal-type graph for the sylvester and Baxter monoids are the same as for the hypoplactic monoid: it is
the notion of isomorphism that has to be modified. This paper develops an abstract framework for this modified notion of
isomorphism (\fullref{Section}{sec:abstractshapes}), and then applies it to the sylvester and Baxter monoids
(\fullref{Sections}{sec:sylvcrystal} and \ref{sec:baxtcrystal}).

This framework is developed from a slightly different perspective from the authors' earlier paper on the hypoplactic
monoid. Essentially, the difference is that rather than having finitely many Kashiwara and `quasi-Kashiwara' operators,
countably many are allowed, and the resulting monoids have infinite rank rather than finite rank. This is a step away
from the representation-theoretic origin of crystal bases, but is actually a much more natural context for a
combinatorially-focussed study. More importantly, it seems to be the correct starting-point to crystallize other monoids
connected with combinatorics, such as sylvester and Baxter monoids. In order to discuss the advantages of these modified
graphs and to clarify what properties a `crystal structure' for the sylvester and Baxter monoids should enjoy,
\fullref{Section}{sec:crystals} briefly recapitulates some of the existing theory. Although some of this material is in existing
literature, the important differences between the finite- and infinite-rank cases are often glossed over.

The resulting crystal-type structure for the sylvester and Baxter monoids turns out to interact neatly with the
sylvester and Baxter versions of the Robinson--Schensted correspondence, and yield new results.

\section{Preliminaries and notation}
\label{sec:preliminaries}

Much of the notation used in this paper is drawn from \cite{cm_hypoplactic}, to which this paper is a sequel. For
definitions and notation for alphabets, words, presentations, standard words, standardization, partitions, compositions,
and weight see \cite[\S~2]{cm_hypoplactic}. For Young tableaux, the Schensted insertion algorithm, and the definition of
the plactic monoid $\plac$ and rank-$n$ plactic monoid $\plac_n$ using these, see \cite[\S~3.1]{cm_hypoplactic};
note however that the present paper uses $\pplac{u}$ and $\qplac{u}$ respectively for the Young tableau and recording tableau
computed from $u \in \aA^*$ and calls these the plactic $\plit$- and $\qlit$-symbols of $u$. Similarly $\phypo{u}$ and
$\qhypo{u}$ are respectively the quasi-ribbon tableau and recording ribbon computed from $u \in \aA^*$ and calls these
the hypoplactic $\plit$- and $\qlit$-symbols of $u$. This is to keep notation uniform with the analogous notions for
the sylvester and Baxter monoids.

Note that the defining relations $\drel{R}_\plac$ given in \cite[\S~1]{cm_hypoplactic} (known as the Knuth relations)
are the reverse of the ones given in \cite{cgm_crystal}. This is because, in the context of crystal bases, the
convention for tensor products gives rise to a `plactic monoid' that is actually anti-isomorphic to the usual notion of
plactic monoid. Since \cite{cm_hypoplactic} and the present paper are mainly concerned with combinatorics, rather than
representation theory, they follow Shimozono \cite{shimozono_crystals} in using the convention that is compatible with
the usual notions of Young tableaux and the Robinston--Schensted correspondence.

For quasi-ribbon tableaux and the Schensted insertion algorithm, and the definition of the hypoplactic monoid $\hypo$
and rank-$n$ hypoplactic monoid $\hypo_n$ using these, see \cite[\S~4]{cm_hypoplactic}.

\section{Infinite-rank crystals and quasi-crystals}
\label{sec:crystals}

This section recalls the definitions of the Kashiwara and quasi-Kashiwara operators, defines the infinite-rank crystal
and quasi-crystal graphs, and shows that isomorphisms between components of these graphs give rise, respectively to the
plactic and hypoplactic monoids. Many of the proofs in this section are not given in full, since they are
straightforward modifications of the proofs for finite-rank crystal and quasi-crystal graphs given in
\cite[\S\S~3~and~5]{cm_hypoplactic}.

\subsection{Crystal graph and the plactic monoid}

Let $i \in \nset$ and define the \defterm{Kashiwara operators} $\ke_i$ and $\kf_i$ on $\aA$ as follows:
\begin{align*}
\ke_i(i+1) &= i, &&\text{$\ke_i(x)$ is undefined for $x \neq i+1$;} \\
\kf_i(i)&=i+1, &&\text{$\kf_i(x)$ is undefined for $x \neq i$.}
\end{align*}
The definition is extended to $\aA^* \setminus \aA$ by the recursion
\begin{align*}
  \ke_i(uv) &= \begin{cases}
    \ke_i(u)\,v & \text{if $\kecount_i(u) > \kfcount_i(v)$;} \\
    u\, \ke_i(v) & \text{if $\kecount_i(u) \leq \kfcount_i(v)$,}
  \end{cases}\displaybreak[0]\\
  \kf_i(uv) &= \begin{cases}
    \kf_i(u)\,v & \text{if $\kecount_i(u) \geq \kfcount_i(v)$;} \\
    u\,\kf_i(v) & \text{if $\kecount_i(u) < \kfcount_i(v)$,}
  \end{cases}
\end{align*}
where $\kecount_i$ and $\kfcount_i$ are auxiliary maps defined by
\begin{align*}
  \kecount_i(w) & = \max\gset[\big]{k \in \nset\cup\set{0}}{\text{$\underbrace{\ke_i\cdots\ke_i}_{\text{$k$ times}}(w)$ is defined}}; \\
  \kfcount_i(w) & = \max\gset[\big]{k \in \nset\cup\set{0}}{\text{$\underbrace{\kf_i\cdots\kf_i}_{\text{$k$ times}}(w)$ is defined}}.
\end{align*}
(Note that this definition is in a sense the mirror image of \cite[Theorem~1.14]{kashiwara_classical}, because of the
choice of definition for readings of tableaux used in this paper. Thus the definition of $\ke_i$ and $\kf_i$ is the same
as in \cite{cm_hypoplactic} and \cite[p.~8]{shimozono_crystals}.)

The operators $\ke_i$ and $\kf_i$ respectively increase and decrease weight whenever they are defined, in the sense that
if $\ke_i(u)$ is defined, then $\wt{\ke_i(u)} > \wt{u}$, and if $\kf_i(u)$ is defined, then $\wt{\kf_i(u)} <
\wt{u}$.
Thus $\ke_i$ and $\kf_i$ are respectively called the Kashiwara \defterm{raising} and \defterm{lowering}
operators. Furthermore, a word on which none of the $\ke_i$ is defined is said to be \defterm{highest-weight}.

The \defterm{crystal graph} for $\plac$, denoted $\Gamma(\plac)$, is the directed labelled graph with vertex set $\aA^*$
and, for $u,v \in \aA^*$, an edge from $u$ to $v$ labelled by $i$ if and only if $v = \kf_i(u)$ (or, equivalently,
$u = \ke_i(v)$). \fullref{Figure}{fig:gammaplac2111} shows part of the crystal graph $\Gamma(\plac)$. For any
$w \in \aA^*$, let $\Gamma(\plac,w)$ denote the connected component of $\Gamma(\plac)$ that contains the vertex $w$. For
$n \in \nset$, the crystal graph for $\plac_n$, denoted $\Gamma(\plac_n)$, is the subgraph of $\Gamma(\plac)$ induced by
$\aA_n^*$. Notice that edge labels in $\Gamma(\plac_n)$ must lie in $\set{1,\ldots,n-1}$, since for $i \geq n$, if
$\kf_i(u)$ is defined, then at least one of $u$ and $\kf_i(u)$ does not lie in $\aA_n^*$.

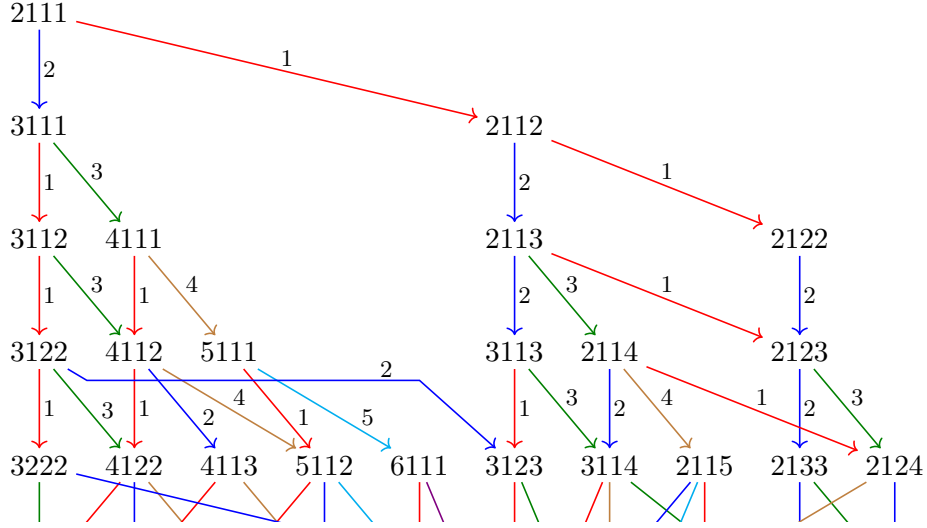
\begin{figure}[t]
  \centering
  \begin{tikzpicture}[bigcrystal,x=12.5mm,y=15mm,labelledcolouredcrystaledges]
    \begin{scope}
      \node (2111) at (0,-0) {$2111$};
      \node (2112) at (5,-1) {$2112$};
      \node (3111) at (0,-1) {$3111$};
      \node (2113) at (5,-2) {$2113$};
      \node (2122) at (8,-2) {$2122$};
      \node (3112) at (0,-2) {$3112$};
      \node (4111) at (1,-2) {$4111$};
      \node (2114) at (6,-3) {$2114$};
      \node (2123) at (8,-3) {$2123$};
      \node (3113) at (5,-3) {$3113$};
      \node (3122) at (0,-3) {$3122$};
      \node (4112) at (1,-3) {$4112$};
      \node (5111) at (2,-3) {$5111$};
      \node (2115) at (7,-4) {$2115$};
      \node (2124) at (9,-4) {$2124$};
      \node (2133) at (8,-4) {$2133$};
      \node (3114) at (6,-4) {$3114$};
      \node (3123) at (5,-4) {$3123$};
      \node (3222) at (0,-4) {$3222$};
      \node (4113) at (2,-4) {$4113$};
      \node (4122) at (1,-4) {$4122$};
      \node (5112) at (3,-4) {$5112$};
      \node (6111) at (4,-4) {$6111$};
      \coordinate (4222) at (0,-5);
      \coordinate (4123) at (1,-5);
      \coordinate (5122) at (2,-5);
      \coordinate (5113) at (3,-5);
      \coordinate (6112) at (4,-5);
      \coordinate (7111) at (4.5,-5);
      \coordinate (3223) at (5,-5);
      \coordinate (3124) at (5.5,-5);
      \coordinate (3115) at (6,-5);
      \coordinate (2116) at (6.5,-5);
      \coordinate (2125) at (7,-5);
      \coordinate (4114) at (7.5,-5);
      \coordinate (3133) at (8,-5);
      \coordinate (2134) at (9,-5);
      \draw[f1] (3113) to (3123);
      \draw[f3,pos=.66] (3122) to (4122);
      \draw[f1] (4112) to (4122);
      \draw[f2] (2114) to (3114);
      \draw[f3] (3113) to (3114);
      \draw[f1,pos=.75] (5111) to (5112);
      \draw[f4] (4112) to (5112);
      \draw[f1] (3112) to (3122);
      \draw[f4] (2114) to (2115);
      \draw[f3] (2113) to (2114);
      \draw[f1] (3111) to (3112);
      \draw[f3] (2123) to (2124);
      \draw[f2] (2123) to (2133);
      \draw[f2] (2112) to (2113);
      \draw[f2,pos=.75] (4112) to (4113);
      \draw[f2] (2122) to (2123);
      \draw[f2] (2113) to (3113);
      \draw[f5,pos=.75] (5111) to (6111);
      \draw[f3] (3112) to (4112);
      \draw[f3] (3111) to (4111);
      \draw[f4] (4111) to (5111);
      \draw[f2] (2111) to (3111);
      \draw[f1] (4111) to (4112);
      \draw[f1] (3122) to (3222);
      %
      \draw[f2,pos=.9] (3122) -- ($ (3122) + (.5,-.25) $) to ($ (3123) + (-1,.75) $) -- (3123);
      \draw[f1] (2114) to (2124);
      \draw[f1] (2111) to (2112);
      \draw[f1] (2112) to (2122);
      \draw[f1] (2113) to (2123);
    \end{scope}
    \begin{scope}[colouredcrystaledges]
      \draw[f3] (3222) to (4222);
      \draw[f4] (4122) to (5122);
      \draw[f1] (4122) to (4222);
      \draw[f2] (4122) to (4123);
      \draw[f1] (4113) to (4123);
      \draw[f4] (4113) to (5113);
      \draw[f1] (5112) to (5122);
      \draw[f5] (5112) to (6112);
      \draw[f2] (5112) to (5113);
      \draw[f1] (6111) to (6112);
      \draw[f6] (6111) to (7111);
      \draw[f3] (3123) to (3124);
      \draw[f1] (3123) to (3223);
      \draw[f4] (3114) to (3115);
      \draw[f3] (3114) to (4114);
      \draw[f1] (3114) to (3124);
      \draw[f5] (2115) to (2116);
      \draw[f2] (2115) to (3115);
      \draw[f2] (2133) to (3133);
      \draw[f3] (2133) to (2134);
      \draw[f4] (2124) to (2125);
      \draw[f2] (2124) to (2134);
      %
      \draw[f1] (2115) to (2125);
      \draw[f2] (3222) to (3223);
    \end{scope}
    \fill[color=white,path fading=fade down] (-.4,-4.5) rectangle (10.4,-5);
    \fill[color=white] (-.4,-4.99) rectangle (10.4,-5.1);
  \end{tikzpicture}
  \caption{Part of the connected component $\Gamma(\plac,2111)$.}
  \label{fig:gammaplac2111}
\end{figure}

Define a relation $\placisom$ on $\aA^*$ as follows: $u \placisom v$ if and only if there is a labelled digraph isomorphism
$\theta : \Gamma(\plac,u) \to \Gamma(\plac,v)$ with $\theta(u) = v$. That is, $u \placisom v$ if $u$ and $v$ lie in corresponding
places in isomorphic components of $\Gamma(\plac)$.

It is worth emphasizing that the definition of $\placisom$ is the crucial point of divergence from the standard
approach to crystals: normally, as in \cite[\S~2.4]{cgm_crystal}, $\placisom$ is defined in terms of
\emph{weight-preserving} isomorphisms.

\begin{proposition}[{\cite[Proposition~2.6]{cgm_crystal}}]
The relation $\placisom$ is a congruence on $\aA^*$.
\end{proposition}

(Strictly speaking, the preceding result follows by modifying the proof of \cite[Proposition~2.6]{cgm_crystal} by
replacing $\aA_n$ with $\aA$ and ignoring the discussion of weight.)

\begin{theorem}[{\cite[Theorem~5.5.1]{lothaire_algebraic}}]
  \label{thm:kekfcompatible}
  Let $u,v \in \aA^*$.
  \begin{enumerate}
  \item $\qplac{u} = \qplac{\ke_i(u)}$ if $\ke_i(u)$ is defined and similarly for $\kf_i$.
  \item Suppose $u \placcong v$. Then $\ke_i(u)$ is defined if and only if $\ke_i(v)$ is defined, in which case
    $\ke_i(u) \placcong \ke_i(v)$ and similarly for $\kf_i$.
  \end{enumerate}
\end{theorem}

(To be strict, \cite[Theorem~5.5.1]{lothaire_algebraic} is concerned only with words in $\aA_n^*$ and
$i \in \set{1,\ldots,n-1}$. To obtain the result as stated here, one simply chooses $n$ to be bigger than the maximum
symbol that appears in $u$ or $v$.)

\begin{proposition}
  \label{prop:placcongsimplac}
  The relations $\placcong$ and $\placisom$ coincide.
\end{proposition}

\begin{proof}
  Let $u \in \aA^*$ and $i \in \nset$. By iterated application of \fullref[(2)]{Theorem}{thm:kekfcompatible}, if
  $u,v \in \aA^*$ are such that $u \placcong v$, then there is a labelled digraph isomorphism
  $\theta : \Gamma(\plac,u) \to \Gamma(\plac,v)$ with $\theta(u) = v$.

  For the converse, let $u,v \in \aA^*$ and suppose there is a labelled digraph isomorphism
  $\theta : \Gamma(\plac,u) \to \Gamma(\plac,v)$ with $\theta(u) = v$. If $\ke_i(u)$ is defined for some $i \in \nset$,
  so is $\ke_i(v)$ (since $\theta$ is a labelled digraph isomorphism); thus one can replace $u$ and $v$ with $\ke_i(u)$
  and $\ke_i(v)$. Eventually one obtains highest weight words $u'$ and $v'$. By the proof of
  \cite[Proposition~5.5.2]{lothaire_algebraic}, $u'$ and $v'$ are Yamanouchi words.

  Suppose $n$ is the largest symbol appearing in $u'$. Then $|u'|_n$ is the length of the longest directed path in
  $\Gamma(\plac)$ made up of edges labelled by $n$ starting at $u'$, or equivalently, $\kfcount_n(u')$. Furthermore,
  $|u'|_{n-1} = |\kf_n^{|u'|_n}(u')|_{n-1}$ is equal to $\kfcount_{n-1}(\kf_n^{|u'|_n}(u'))$. Continuing in this way, one
  sees that for each $i \in \set{1,\ldots,n}$
  \[
  |u'|_i = |\kf_{i+1}^{|u'|_{i+1}}\cdots \kf_{n}^{|u'|_{n}}(u')|_i = \kfcount_i\parens[\big]{\kf_{i+1}^{|u'|_{i+1}}\cdots \kf_{n}^{|u'|_{n}}(u')}.
  \]
  Hence $\wt{u'}$ is determined by the connected component $\Gamma(\plac,u')$. Similarly, $\wt{v'}$ is determined by the
  connected component $\Gamma(\plac,v')$.

  Since the components $\Gamma(\plac,u')$ and $\Gamma(\plac,v')$ are isomorphic, $\wt{u'} = \wt{v'}$. Thus, since $u'$ and
  $v'$ are Yamanouchi words, $\pplac{u'}$ and $\pplac{v'}$ are both equal to the tableaux of shape $\wt{u'} = \wt{v'}$ containing
  only symbols $i$ on row $i$. Hence $u' \placcong v'$. By iterated application of
  \fullref[(2)]{Theorem}{thm:kekfcompatible}, $w \placcong \theta(w)$ for all $w \in \Gamma(\plac,u)$.
\end{proof}

In the finite rank case, the analogy of \fullref{Proposition}{prop:placcongsimplac} does not hold: for example, the
connected components $\Gamma(\plac_n,\emptyword)$ and $\Gamma(\plac_n,n(n-1)\cdots 21)$ both consist of single vertices
and thus are isomorphic as labelled digraphs. (Thus \cite[Proposition~5.5.2]{lothaire_algebraic} is technically false as
stated.) This is the crucial difference between the infinite-rank and finite-rank cases: In the infinite-rank case, all
necessary information about weights is contained in the abstract graph $\Gamma(\plac)$ and so it suffices to consider
labelled digraph isomorphisms between connected components. However, in the finite-rank case, the abstract graph
$\Gamma(\plac_n)$ does not contain sufficient information to define the plactic monoid $\plac_n$, and so it is necessary
to restrict to weight-preserving labelled digraph isomorphisms.

\begin{remark}
  \label{rem:slncrystals}
  In fact, the abstract graph $\Gamma(\plac_n)$ is the crystal graph for $\mathfrak{sl}_n$. This corresponds to the fact
  that quotienting by the full column $n(n-1)\cdots 21$ projects from isomorphism classes of $\mathfrak{gl}_n$ crystal
  graphs to isomorphism classes of $\mathfrak{sl}_n$ crystal graphs. The only extra information one requires to reverse
  this projection is how many full columns are present.
\end{remark}

\subsection{Quasi-crystal graph and the hypoplactic monoid}

Let $i \in \nset$. Define the \defterm{quasi-Kashiwara operators} $\e_i$ and $\f_i$ on $\aA^*$ as
follows: Let $u \in \aA^*$.
\begin{itemize}
\item If $u$ contains a subsequence $(i+1)i$, both $\e_i(u)$ and $\f_i(u)$ are undefined.
\item If $u$ does not contain a subsequence $(i+1)i$, but $u$ contains at least one symbol $i+1$, then $\e_i(u)$ is the
  word obtained from $u$ by replacing the left-most symbol $i+1$ by $i$; if $u$ contains no symbol $i+1$, then $\e_i(u)$
  is undefined.
\item If $u$ does not contain a subsequence $(i+1)i$, but $u$ contains at least one symbol $i$, then $\f_i(u)$ is the
  word obtained from $u$ by replacing the right-most symbol $i$ by $i+1$; if $u$ contains no symbol $i$, then $\f_i(u)$
  is undefined.
\end{itemize}
(As in \cite{cm_hypoplactic}, a subsequence may be made up of non-consecutive letters; thus the phrase `contains a
subsequence $(i+1)i$' is equivalent to `contains a symbol $i+1$ somewhere to the left of a symbol $i$'.) For example,
\begin{align*}
\e_2(3123) &\;\text{is undefined since $3123$ contains the subsequence $32$;} \\
\f_2(3131) &\;\text{is undefined since $3131$ does not contains a symbol $2$;} \\
\f_1(3113) &= 3123.
\end{align*}

\begin{lemma}[{\cite[Lemma~1]{cm_hypoplactic}}]
  \label{lem:efinverse}
  For all $i \in \nset$, the operators $\e_i$ and $\f_i$ are mutually inverse, in the sense that if
  $\e_i(u)$ is defined, $u = \f_i(\e_i(u))$, and if $\f_i(u)$ is defined, $u = \e_i(\f_i(u))$.
\end{lemma}

The operators $\e_i$ and $\f_i$ respectively increase and decrease weight whenever they are defined, in the sense that
if $\e_i(u)$ is defined, then $\wt{\e_i(u)} > \wt{u}$, and if $\f_i(u)$ is defined, then $\wt{\f_i(u)} < \wt{u}$. Thus
$\e_i$ and $\f_i$ are respectively called the quasi-Kashiwara \defterm{raising} and \defterm{lowering}
operators. Furthermore, a word on which none of the $\e_i$ is defined is said to be \defterm{highest-weight}. (Strictly
speaking, it is necessary to distinguish words that are highest-weight with respect to the operators $\e_i$ and words
that are highest-weight with respect to the operators $\ke_i$. Since the $\ke_i$ are not mentioned in the remainder of
the paper, the term `highest-weight' will always be with respect to the operators $\e_i$.)

The \defterm{quasi-crystal graph} for $\hypo$, denoted $\Gamma(\hypo)$, is the directed labelled graph with vertex set $\aA^*$
and, for $u,v \in \aA^*$, an edge from $u$ to $v$ labelled by $i$ if and only if $v = \kf_i(u)$ (or, equivalently,
$u = \ke_i(v)$). \fullref{Figure}{fig:gammahypo1212} shows part of the crystal graph $\Gamma(\hypo)$. For any
$w \in \aA^*$, let $\Gamma(\hypo,w)$ denote the connected component of $\Gamma(\hypo)$ that contains the vertex
$w$. The quasi-crystal graph for $\hypo_n$, denoted $\Gamma(\hypo_n)$, is the subgraph of $\Gamma(\hypo)$ induced by
$\aA_n^*$. Notice that edge labels in $\Gamma(\hypo_n)$ must lie in $\set{1,\ldots,n-1}$, since for $i \geq n$, if
$\f_i(u)$ is defined, then at least one of $u$ and $\f_i(u)$ does not lie in $\aA_n^*$.

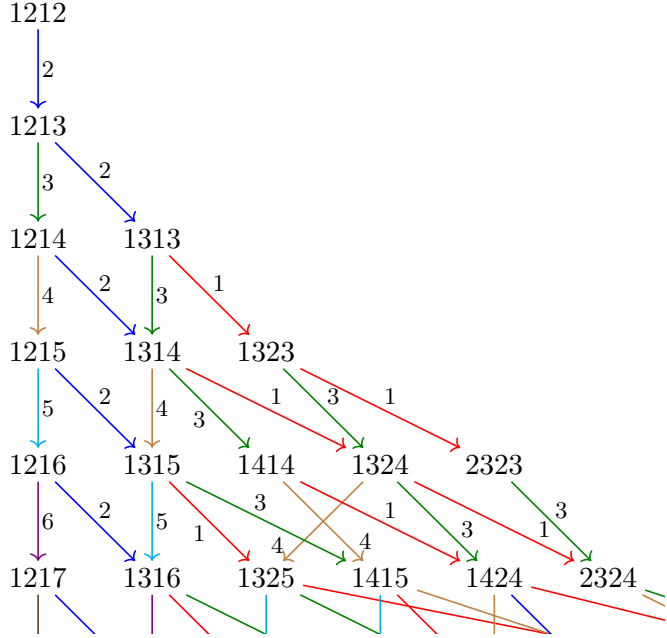
\begin{figure}[t]
  \centering
  \begin{tikzpicture}[bigcrystal,x=15mm,y=15mm,labelledcolouredcrystaledges]
    \clip (-.5,.5) -- (6.5,-.5) -- (6.5,-6) -- (-.5,-6) -- cycle;
    \begin{scope}
      \node (1212) at (0,-0) {$1212$};
      \node (1213) at (0,-1) {$1213$};
      \node (1214) at (0,-2) {$1214$};
      \node (1313) at (1,-2) {$1313$};
      \node (1215) at (0,-3) {$1215$};
      \node (1314) at (1,-3) {$1314$};
      \node (1323) at (2,-3) {$1323$};
      \node (1216) at (0,-4) {$1216$};
      \node (1315) at (1,-4) {$1315$};
      \node (1324) at (3,-4) {$1324$};
      \node (1414) at (2,-4) {$1414$};
      \node (2323) at (4,-4) {$2323$};
      \node (1217) at (0,-5) {$1217$};
      \node (1316) at (1,-5) {$1316$};
      \node (1325) at (2,-5) {$1325$};
      \node (1415) at (3,-5) {$1415$};
      \node (1424) at (4,-5) {$1424$};
      \node (2324) at (5,-5) {$2324$};
      \coordinate (1218) at (0,-6);
      \coordinate (1317) at (1,-6);
      \coordinate (1326) at (2,-6);
      \coordinate (1416) at (3,-6);
      \coordinate (1425) at (4,-6);
      \coordinate (1434) at (5,-6);
      \coordinate (1515) at (6,-6);
      \coordinate (2325) at (7,-6);
      \coordinate (2424) at (8,-6);
      \draw[f5] (1315) to (1316);
      \draw[f5] (1215) to (1216);
      \draw[f3] (2323) to (2324);
      \draw[f2] (1213) to (1313);
      \draw[f4] (1214) to (1215);
      \draw[f3,pos=.75] (1324) to (1424);
      \draw[f3] (1313) to (1314);
      \draw[f2] (1214) to (1314);
      \draw[f4] (1314) to (1315);
      \draw[f4,swap,pos=.95] (1324) to (1325);
      \draw[f6] (1216) to (1217);
      \draw[f2] (1215) to (1315);
      \draw[f1,pos=.75] (1324) to (2324);
      \draw[f1,swap] (1315) to (1325);
      \draw[f1] (1323) to (2323);
      \draw[f3] (1323) to (1324);
      \draw[f2] (1216) to (1316);
      \draw[f1] (1314) to (1324);
      \draw[f3,swap] (1314) to (1414);
      \draw[f1] (1313) to (1323);
      \draw[f3] (1213) to (1214);
      \draw[f1,pos=.5] (1414) to (1424);
      \draw[f4,pos=.9] (1414) to (1415);
      \draw[f3,pos=.4] (1315) to (1415);
      \draw[f2] (1212) to (1213);
    \end{scope}
    \begin{scope}[colouredcrystaledges]
      \draw[f7] (1217) to (1218);
      \draw[f2] (1217) to (1317);
      \draw[f3] (1316) to (1416);
      \draw[f1] (1316) to (1326);
      \draw[f6] (1316) to (1317);
      \draw[f5] (1325) to (1326);
      \draw[f3] (1325) to (1425);
      \draw[f1] (1325) to (2325);
      \draw[f5] (1415) to (1416);
      \draw[f1] (1415) to (1425);
      \draw[f4] (1415) to (1515);
      \draw[f4] (1424) to (1425);
      \draw[f2] (1424) to (1434);
      \draw[f1] (1424) to (2424);
      \draw[f4] (2324) to (2325);
      \draw[f3] (2324) to (2424);
    \end{scope}
    \fill[color=white,path fading=fade right] (5.5,-6.5) rectangle (6.5,-5);
    \fill[color=white,path fading=fade down] (-.3,-6) rectangle (6.5,-5.5);
    \fill[color=white] (-.3,-5.99) rectangle (6.5,-6.1);
  \end{tikzpicture}
  \caption{Part of the connected component $\Gamma(\hypo,1212)$.}
  \label{fig:gammahypo1212}
\end{figure}

Define a relation $\hypoisom$ on $\aA^*$ as follows: $u \sim v$ if and only if there is a labelled digraph isomorphism
$\theta : \Gamma(\hypo,u) \to \Gamma(\hypo,v)$ with $\theta(u) = v$. That is, $u \hypoisom v$ if $u$ and $v$ lie in
corresponding places in isomorphic components of $\Gamma(\hypo)$. For instance, as can be seen in
\fullref{Figure}{fig:gammahypothree}, $1241 \hypoisom 2141$.

Again, note that the definition of $\hypoisom$ is the point of divergence from the previous definition in
\cite{cm_hypoplactic}, where $\hypoisom$ is defined in terms of \emph{weight-preserving} isomorphisms.

By the proof of \cite[Proposition~4]{cm_hypoplactic}, replacing $\aA_n$ with $\aA$ and ignoring the discussion of
weight, the relation $\hypoisom$ is a congruence on $\aA^*$.

\begin{proposition}[{\cite[Proposition~5]{cm_hypoplactic}}]
  \label{prop:eipreservesstd}
  Let $u \in \aA^*$ and $i \in \nset$.
  \begin{enumerate}
  \item If $\e_i(u)$ is defined, then $\std{\e_i(u)} = \std{u}$.
  \item If $\f_i(u)$ is defined, then $\std{\f_i(u)} = \std{u}$.
  \end{enumerate}
\end{proposition}

In contrast to \fullref{Proposition}{prop:eipreservesstd}, the usual Kashiwara operators do not preserve
standardization: $\ke_2(1322) = 1323$, but $\std{1322} = 1423 \neq 1324 = \std{1323}$.


\begin{proposition}
  \label{prop:hypocongsimhypo}
  The relations $\hypocong$ and $\hypoisom$ coincide.
\end{proposition}

\begin{proof}
  Let $u,v \in \aA^*$.

  Suppose that $u \hypocong v$. Using the reasoning that leads to the proof of \cite[Proposition~8]{cm_hypoplactic},
  but replacing $\aA_n$ by $\aA$, proves that $u \hypoisom v$.

  Now suppose that there is a labelled digraph isomorphism $\theta : \Gamma(\hypo,u) \to \Gamma(\hypo,v)$ with
  $\theta(u) = v$. Then there is a sequence of operators $\e_{i_1},\ldots,\e_{i_k}$ such that the words
  $u' = \e_{i_1}\cdots\e_{i_k}(u)$ and $v' = \e_{i_1}\cdots\e_{i_k}(v)$ are highest-weight. By the same reasoning as in
  \fullref{Proposition}{prop:placcongsimplac}, $\wt{u'}$ and $\wt{v'}$ are determined by the isomorphic connected
  components $\Gamma(\hypo,u)$ and $\Gamma(\hypo,v)$, and thus $\wt{u'} = \wt{v'}$. Since
  $u' \hypocong \colreading{\phypo{u'}}$, and $v' \hypocong \colreading{\phypo{v'}}$ by \cite[Lemmata~5 and
  6]{cm_hypoplactic} (replacing $\aA_n$ by $\aA$), and $\colreading{\phypo{u'}}$ and $\colreading{\phypo{v'}}$ are
  highest-weight (since ${\hypocong} \subseteq {\sim_\hypo}$ by the first part) and have the same weight,
  $\colreading{\phypo{u'}} = \colreading{\phypo{v'}}$ by \cite[Corollary~3]{cm_hypoplactic}. Hence
  $u' \hypocong \colreading{\phypo{u'}} = \colreading{\phypo{v'}} \hypocong v'$. By \cite[Lemma~6(2)]{cm_hypoplactic},
  applying $\f_i$ to $\hypocong$-related words yields $\hypocong$-related words. Hence $u \hypocong v$.
\end{proof}

\begin{figure}[t]
  \centering
  \begin{tikzpicture}[bigcrystal,x=15mm,y=15mm,labelledcolouredcrystaledges]
    \begin{scope}[xshift=0mm]
      \begin{scope}
        \node (1212) at (0,-0) {$1212$};
        \node (1213) at (0,-1) {$1213$};
        \node (1214) at (0,-2) {$1214$};
        \node (1313) at (1,-2) {$1313$};
        \node (1215) at (0,-3) {$1215$};
        \node (1314) at (1,-3) {$1314$};
        \node (1323) at (2,-3) {$1323$};
        \coordinate (1216) at (0,-4) {};
        \coordinate (1315) at (1,-4) {};
        \coordinate (1324) at (3,-4) {};
        \coordinate (1414) at (2,-4) {};
        \coordinate (2323) at (4,-4) {};
        \draw[f2] (1212) to (1213);
        \draw[f2] (1213) to (1313);
        \draw[f3] (1213) to (1214);
        \draw[f4] (1214) to (1215);
        \draw[f2] (1214) to (1314);
        \draw[f3] (1313) to (1314);
        \draw[f1] (1313) to (1323);
      \end{scope}
      \begin{scope}[colouredcrystaledges]
        \draw[f5] (1215) to (1216);
        \draw[f2] (1215) to (1315);
        \draw[f4] (1314) to (1315);
        \draw[f1] (1314) to (1324);
        \draw[f3] (1314) to (1414);
        \draw[f1] (1323) to (2323);
        \draw[f3] (1323) to (1324);
      \end{scope}
      \fill[color=white,path fading=fade down] ($ (1215) + (-.2,-.5) $) rectangle ($ (1323) + (2,-.8) $);
      \fill[color=white] ($ (1215) + (-.2,-.79) $) rectangle ($ (1323) + (2,-1.1) $);
      \fill[color=white,path fading=fade right] ($ (1323) + (.3,0) $) rectangle ($ (1323) + (.6,-1) $);
      \fill[color=white] ($ (1323) + (.58,0) $) rectangle ($ (1323) + (2,-1) $);
    \end{scope}
    \begin{scope}[xshift=45mm]
      \begin{scope}
        \node (2121) at (0,-0) {$2121$};
        \node (2131) at (0,-1) {$2131$};
        \node (2141) at (0,-2) {$2141$};
        \node (3131) at (1,-2) {$3131$};
        \node (2151) at (0,-3) {$2151$};
        \node (3141) at (1,-3) {$3141$};
        \node (3231) at (2,-3) {$3231$};
        \coordinate (2161) at (0,-4) {};
        \coordinate (3151) at (1,-4) {};
        \coordinate (3241) at (3,-4) {};
        \coordinate (4141) at (2,-4) {};
        \coordinate (3232) at (4,-4) {};
        \draw[f2] (2121) to (2131);
        \draw[f2] (2131) to (3131);
        \draw[f3] (2131) to (2141);
        \draw[f4] (2141) to (2151);
        \draw[f2] (2141) to (3141);
        \draw[f3] (3131) to (3141);
        \draw[f1] (3131) to (3231);
      \end{scope}
      \begin{scope}[colouredcrystaledges]
        \draw[f5] (2151) to (2161);
        \draw[f2] (2151) to (3151);
        \draw[f4] (3141) to (3151);
        \draw[f1] (3141) to (3241);
        \draw[f3] (3141) to (4141);
        \draw[f1] (3231) to (3232);
        \draw[f3] (3231) to (3241);
      \end{scope}
      \fill[color=white,path fading=fade down] ($ (2151) + (-.2,-.5) $) rectangle ($ (3231) + (2,-.8) $);
      \fill[color=white] ($ (2151) + (-.2,-.79) $) rectangle ($ (3231) + (2,-1.1) $);
      \fill[color=white,path fading=fade right] ($ (3231) + (.3,0) $) rectangle ($ (3231) + (.6,-1) $);
      \fill[color=white] ($ (3231) + (.58,0) $) rectangle ($ (3231) + (2,-1) $);
    \end{scope}
    \begin{scope}[xshift=90mm]
      \begin{scope}
        \node (1221) at (0,-0) {$1221$};
        \node (1231) at (0,-1) {$1231$};
        \node (1241) at (0,-2) {$1241$};
        \node (1331) at (1,-2) {$1331$};
        \node (1251) at (0,-3) {$1251$};
        \node (1341) at (1,-3) {$1341$};
        \node (1332) at (2,-3) {$1332$};
        \coordinate (1261) at (0,-4) {};
        \coordinate (1351) at (1,-4) {};
        \coordinate (1342) at (3,-4) {};
        \coordinate (1441) at (2,-4) {};
        \coordinate (2332) at (4,-4) {};
        \draw[f2] (1221) to (1231);
        \draw[f2] (1231) to (1331);
        \draw[f3] (1231) to (1241);
        \draw[f4] (1241) to (1251);
        \draw[f2] (1241) to (1341);
        \draw[f3] (1331) to (1341);
        \draw[f1] (1331) to (1332);
      \end{scope}
      \begin{scope}[colouredcrystaledges]
        \draw[f5] (1251) to (1261);
        \draw[f2] (1251) to (1351);
        \draw[f4] (1341) to (1351);
        \draw[f1] (1341) to (1342);
        \draw[f3] (1341) to (1441);
        \draw[f1] (1332) to (2332);
        \draw[f3] (1332) to (1342);
      \end{scope}
      \fill[color=white,path fading=fade down] ($ (1251) + (-.2,-.5) $) rectangle ($ (1332) + (2,-.8) $);
      \fill[color=white] ($ (1251) + (-.2,-.79) $) rectangle ($ (1332) + (2,-1.1) $);
      \fill[color=white,path fading=fade right] ($ (1332) + (.3,0) $) rectangle ($ (1332) + (.6,-1) $);
      \fill[color=white] ($ (1332) + (.58,0) $) rectangle ($ (1332) + (2,-1) $);
    \end{scope}
    \node at ($ (1212) + (15mm,7.5mm) $) {$\overbrace{\hbox{\vrule width 40mm height 0cm depth 0cm}}$};
    \node at ($ (2121) + (15mm,7.5mm) $) {$\overbrace{\hbox{\vrule width 40mm height 0cm depth 0cm}}$};
    \node at ($ (1221) + (15mm,7.5mm) $) {$\overbrace{\hbox{\vrule width 40mm height 0cm depth 0cm}}$};
    \node at ($ (1212) + (15mm,12.5mm) $) {$\Gamma(\hypo,1212)$};
    \node at ($ (2121) + (15mm,12.5mm) $) {$\Gamma(\hypo,2121)$};
    \node at ($ (1221) + (15mm,12.5mm) $) {$\Gamma(\hypo,1221)$};
  \end{tikzpicture}
  \caption{The isomorphic components $\Gamma(\hypo,1212)$, $\Gamma(\hypo,2121)$, and $\Gamma(\hypo,1221)$ of the
    quasi-crystal graph $\Gamma(\hypo)$.}
  \label{fig:gammahypothree}
\end{figure}
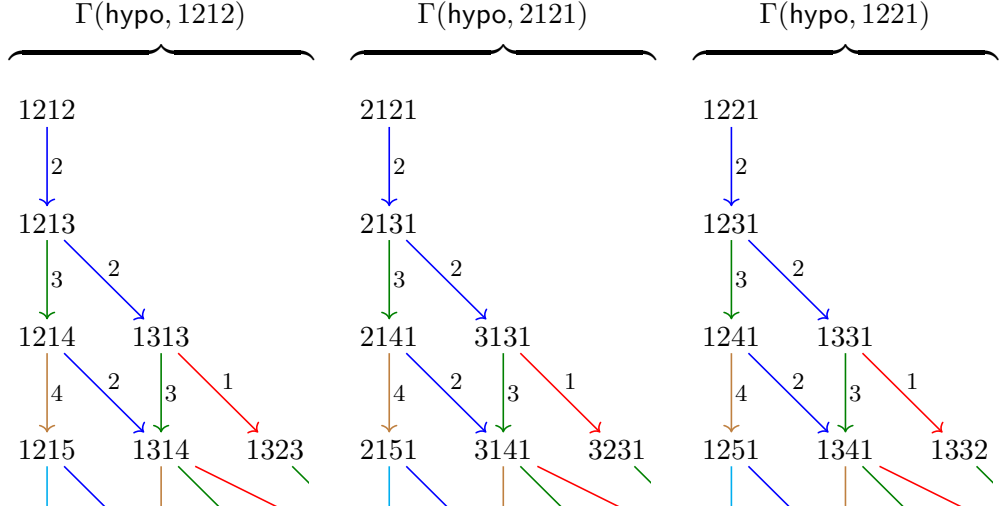

In the finite rank case, the analogy of \fullref{Proposition}{prop:hypocongsimhypo} does not hold: the connected
components $\Gamma(\hypo_n,\emptyword)$ and $\Gamma(\hypo_n,n(n-1)\cdots 21)$ both consist of single vertices and thus
are isomorphic as labelled digraphs. This is again the crucial difference between the infinite-rank and finite-rank
cases: In the infinite-rank case, all necessary information about weights is contained in the abstract graph
$\Gamma(\hypo)$ and so it suffices to consider labelled digraph isomorphisms between connected components. However, in
the finite-rank case, the abstract graph $\Gamma(\hypo_n)$ does not contain sufficient information, and so it is
necessary to restrict to \emph{weight-preserving} labelled digraph isomorphisms.

There is a very neat interaction between the structure of $\Gamma(\hypo)$ and the hypoplactic analogue of the
Robinson--Schensted correspondence; the proof for the finite-rank case applies in the infinite-rank case:

\begin{theorem}[{\cite[Theorem~2]{cm_hypoplactic}}]
  \label{thm:samecomponentiffsameqhypo}
  Let $u,v \in \aA^*$. The words $u$ and $v$ lie in the same connected component of $\Gamma(\hypo)$ if and only if
  $\qhypo{u} = \qhypo{v}$.
\end{theorem}

Thus the recording ribbons $\qhypo{\cdot}$ produced by the insertion algorithm (see the exposition in
\cite[\S~4]{cm_hypoplactic}) index connected components of $\Gamma(\hypo)$. However, in the infinite-rank case,
connected components are also indexed by standard words, as shown in \fullref{Corollary}{corol:stdinsamecomp} below.

\begin{lemma}
  \label{lem:qhypostdqhypo}
  For any $u \in \aA^*$, the equality $\qhypo{\std{u}} = \qhypo{u}$ holds.
\end{lemma}

\begin{proof}
  Let $u = u_1\cdots u_k$ and $\std{u} = s_1\cdots s_k$, where $u_h,s_h \in \aA$. Since
  $\std{u_1\cdots u_k} = \std{s_1\cdots s_k}$, it follows that $\std{u_1\cdots u_h} = \std{s_1\cdots s_h}$ for all $h$
  (this is immediate from the definition of standardization \cite[Lemma~2.2]{novelli_hypoplactic}). Thus, by
  \cite[Theorems~4.12 and~4.16]{novelli_hypoplactic}, $\qhypo{u_1\cdots u_h}$ and $\qhypo{s_1\cdots s_h}$ both have shape
  $\descomp{\std{u_1\cdots u_h}^{-1}}$ for all $h$. The sequence of these shapes determines where new symbols are
  inserted during the computation of $\phypo{u}$ and $\qhypo{\std{u}}$ by \cite[Algorithm~4]{cm_hypoplactic}, and so
  $\qhypo{u} = \qhypo{\std{u}}$.
\end{proof}

\begin{corollary}
  \label{corol:stdinsamecomp}
  In every connected component of $\Gamma(\hypo)$, there is exactly one standard word, and so for all $u,v \in \aA^*$,
  the words $u$ and $v$ lie in the same connected component if and only if $\std{u} = \std{v}$.
\end{corollary}

\begin{proof}
  By \fullref{Lemma}{lem:qhypostdqhypo} and \fullref{Theorem}{thm:samecomponentiffsameqhypo}, $u$ and $\std{u}$ are in
  the same component of $\Gamma(\hypo)$, so each connected component contains at least one standard word. Suppose $u'$
  is a standard word in the same component of $\Gamma(\hypo)$ as $u$. Then it is possible to apply a sequence of
  quasi-Kashiwara operators to $u$ and obtain $u'$. By \fullref{Proposition}{prop:eipreservesstd}, quasi-Kashiwara
  operators preserve standardization; thus $\std{u} = \std{u'}$ and so $u = u'$ since $u$ and $u'$ are standard
  words. The last claim in the statement follows immediately.
\end{proof}

\begin{remark}
  As noted in \cite{cm_hypoplactic}, a related notion of `quasi-crystal' is found in Krob and Thibon
  \cite{krob_noncommutative5}; see also \cite{hivert_hecke}. However, the Krob--Thibon quasi-crystal describes the
  restriction to quasi-ribbon tableaux (or rather corresponding to quasi-ribbon tableaux) of the action of the usual
  Kashiwara operators: it does not apply to all words and so does not yield a definition of $\hypocong$ via isomorphisms
  of components.
\end{remark}

\section{Binary trees and the sylvester and Baxter monoids}
\label{sec:binarytrees}

This section gathers the necessary definitions and background on binary search trees, the sylvester and Baxter monoids,
and the analogues of the Robinson--Schensted correspondence. For further background on the sylvester monoid, see
\cite{hivert_sylvester}; for the Baxter monoid, see \cite{giraudo_baxter2}.

\subsection{Binary trees and insertion}

A \defterm{left strict binary search tree} is a labelled rooted binary tree where the label of each node
is strictly greater than the label of every node in its left subtree, and less or equal to than the label of every
node in its right subtree. A \defterm{right strict binary search tree} is a labelled rooted binary tree where the label of each node is greater
than or equal to the label of every node in its left subtree, and strictly less than the label of every node in its
right subtree. The following are examples of, respectively, left strict and right strict binary search trees:
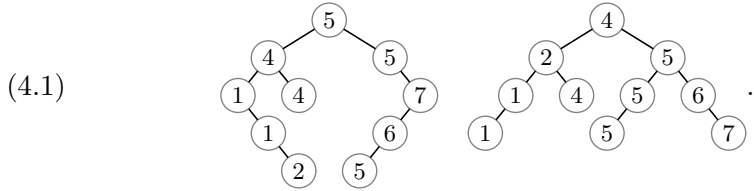
\begin{equation}
\label{eq:bsteg}
\begin{tikzpicture}[tinybst,baseline=-10mm]
  \node (root) {$5$}
    child[sibling distance=16mm] { node (0) {$4$}
      child { node (00) {$1$}
        child[missing]
        child { node (001) {$1$}
          child[missing]
          child { node (0011) {$2$} }
        }
      }
      child { node (01) {$4$} }
    }
    child[sibling distance=16mm] { node (1) {$5$}
      child[missing]
      child { node (11) {$7$}
        child { node (110) {$6$}
          child { node (1100) {$5$} }
          child[missing]
        }
        child[missing]
      }
    };
\end{tikzpicture}
\quad
\begin{tikzpicture}[tinybst,baseline=-10mm]
  \node (root) {$4$}
    child[sibling distance=16mm] { node (0) {$2$}
      child { node (00) {$1$}
        child { node (000) {$1$} }
        child[missing]
      }
      child { node (01) {$4$} }
    }
    child[sibling distance=16mm] { node (1) {$5$}
      child { node (10) {$5$}
        child { node (100) {$5$} }
        child[missing]
      }
      child { node (11) {$6$}
        child[missing]
        child { node (111) {$7$} }
      }
    };
\end{tikzpicture}.
\end{equation}
A binary search tree is \defterm{standard} if the labels of its nodes are exactly the integers from $1$ to the
number of nodes in the tree. The \defterm{shape} of a binary search tree $T$, denoted $\Sh{T}$, is simply its underlying
unlabelled rooted binary tree. The notion of shape will be important later in the paper.

A \defterm{decreasing tree} is a labelled rooted binary tree where the label of each node is greater than the label of
its children. An \defterm{increasing tree} is a labelled rooted binary tree where the label of each node is less than
the label of its children. The following are examples of, respectively, standard increasing and standard decreasing
trees:
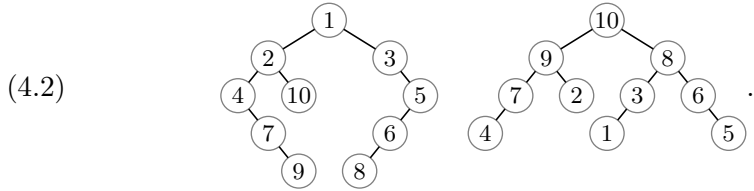
\begin{equation}
\label{eq:incdectreeeg}
\begin{tikzpicture}[tinybst,baseline=-10mm]
  \node (root) {$1$}
    child[sibling distance=16mm] { node (0) {$2$}
      child { node (00) {$4$}
        child[missing]
        child { node (001) {$7$}
          child[missing]
          child { node (0011) {$9$} }
        }
      }
      child { node (01) {$10$} }
    }
    child[sibling distance=16mm] { node (1) {$3$}
      child[missing]
      child { node (11) {$5$}
        child { node (110) {$6$}
          child { node (1100) {$8$} }
          child[missing]
        }
        child[missing]
      }
    };
\end{tikzpicture}
\quad
\begin{tikzpicture}[tinybst,baseline=-10mm]
  \node (root) {$10$}
    child[sibling distance=16mm] { node (0) {$9$}
      child { node (00) {$7$}
        child { node (000) {$4$} }
        child[missing]
      }
      child { node (01) {$2$} }
    }
    child[sibling distance=16mm] { node (1) {$8$}
      child { node (10) {$3$}
        child { node (100) {$1$} }
        child[missing]
      }
      child { node (11) {$6$}
        child[missing]
        child { node (111) {$5$} }
      }
    };
\end{tikzpicture}.
\end{equation}
The \defterm{left-to-right infix traversal} (or simply the \defterm{infix traversal}) of a rooted binary tree $T$ is the
sequence that `visits' every node in the tree as follows: it recursively performs the infix traversal of
the left subtree of the root of $T$, then visits the root of $T$, then recursively performs the infix
traversal of the right subtree of the root of $T$. Thus the infix traversal of any binary tree with the same
shape as the right-hand tree in \eqref{eq:bsteg} visits nodes as follows:
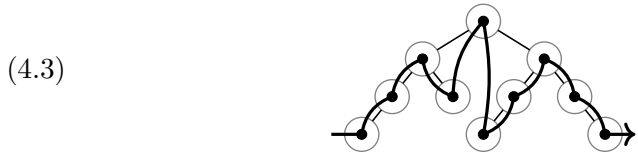
\begin{equation}
\begin{tikzpicture}[tinybst,baseline=-7.5mm]
  \node (root) {}
    child[sibling distance=16mm] { node (0) {}
      child { node (00) {}
        child { node (000) {} }
        child[missing]
      }
      child { node (01) {} }
    }
    child[sibling distance=16mm] { node (1) {}
      child { node (10) {}
        child { node (100) {} }
        child[missing]
      }
      child { node (11) {}
        child[missing]
        child { node (111) {} }
      }
    };
  \begin{scope}[very thick,line cap=round]
    \draw (000.center) edge[bend left=30] (00.center);
    \draw (00.center) edge[bend left=30] (0.center);
    \draw (0.center) edge[bend right=30] (01.center);
    \draw (01.center) edge[bend left=20] (root.center);
    \draw (root.center) edge[bend left=10] (100.center);
    \draw (100.center) edge[bend right=30] (10.center);
    \draw (10.center) edge[bend right=30] (1.center);
    \draw (1.center) edge[bend left=30] (11.center);
    \draw (11.center) edge[bend left=30] (111.center);
  \end{scope}
  \draw[very thick] ($ (000.center) + (-4mm,0) $) -- (000.center);
  \draw[very thick,->] (111.center) -- ($ (111.center) + (4mm,0) $);
  \foreach\x in {000,00,01,0,100,10,111,11,1,root} {
    \draw[draw=black,fill=black] (\x.center) circle (.66mm);
  }
\end{tikzpicture}.
\end{equation}
The \defterm{left-to-right postfix traversal}, or simply the \defterm{postfix traversal}, of a rooted binary tree $T$ is
the sequence that `visits' every node in the tree as follows: it recursively perform the postfix traversal
of the left subtree of the root of $T$, then recursively perform the postfix traversal of the right
subtree of the root of $T$, then visits the root of $T$. Thus the postfix traversal of any binary tree
with the same shape as the right-hand tree in \eqref{eq:bsteg} visits nodes as follows:
\begin{equation}
\begin{tikzpicture}[tinybst,baseline=-7.5mm]
  \node (root) {}
    child[sibling distance=16mm] { node (0) {}
      child { node (00) {}
        child { node (000) {} }
        child[missing]
      }
      child { node (01) {} }
    }
    child[sibling distance=16mm] { node (1) {}
      child { node (10) {}
        child { node (100) {} }
        child[missing]
      }
      child { node (11) {}
        child[missing]
        child { node (111) {} }
      }
    };
  \begin{scope}[very thick,line cap=round]
    \draw (000.center) edge[bend left=30] (00.center);
    \draw (00.center) edge[bend right=30] (01.center);
    \draw (01.center) edge[bend left=30] (0.center);
    \draw (0.center) edge[bend left=40] (100.center);
    \draw (100.center) edge[bend right=30] (10.center);
    \draw (10.center) edge[bend left=10] (111.center);
    \draw (111.center) edge[bend right=30] (11.center);
    \draw (11.center) edge[bend right=30] (1.center);
    \draw (1.center) edge[bend right=30] (root.center);
  \end{scope}
  \draw[very thick] ($ (000.center) + (-4mm,0) $) -- (000.center);
  \draw[very thick,->] (root.center) -- ($ (root.center) + (0,4mm) $);
  \foreach\x in {000,00,01,0,100,10,111,11,1,root} {
    \draw[draw=black,fill=black] (\x.center) circle (.66mm);
  }
\end{tikzpicture}.
\end{equation}
The \defterm{infix reading} $\inreading{T}$ (respectively, \defterm{postfix reading} $\postreading{T}$) of a labelled
binary search tree $T$ is defined to be the word obtained by listing the labels of the nodes visited during the infix
(respectively, postfix) traversal. For example, the infix and postfix readings of the
right-hand tree in \eqref{eq:bsteg} are, respectively, $1124455567$ and $1142557654$. Note that the weights of the
infix and postfix readings of a given binary search tree are the same; both weights are simply the tuple that
describes the number of each symbol in $\aA$ appearing in the tree. Thus define the \defterm{weight} of a labelled
binary search tree to be the weight of these readings.

The following result is immediate from  the definition of a binary search tree, but it is used frequently:

\begin{proposition}
  \label{prop:infixreading}
  For any \textparens{left or right strict} binary search tree $T$, the infix reading $\inreading{T}$ is always the
  unique weakly increasing word with the same weight as $T$.
\end{proposition}

Let $T$ be a rooted binary tree and let $x_1,\ldots,x_m$ be the nodes of $T$ in the order visited in an infix traversal. Let
$i_1,\ldots,i_{k-1}$ be such that the $x_{i_h}$ are precisely the nodes of $T$ that have non-empty left subtrees. The
\defterm{left interval partition} of $T$ is the partition
\[
\set[\Big]{
\underbrace{\set{x_1,\ldots,x_{i_1-1}}}_{\text{$1$st left interval}},
\underbrace{\set{x_{i_1},\ldots,x_{i_2-1}}}_{\text{$2$nd left interval}},
\ldots,
\underbrace{\set{x_{i_{k-2}},\ldots,x_{i_{k-1}-1}}}_{\text{$k-1$-th left interval}},
\underbrace{\set{x_{k-1},\ldots,x_{m}}}_{\text{$k$-th left interval}}
};
\]
the $h$-th part of this partition is the \defterm{$h$-th left interval} of $T$. Let $j_1,\ldots,j_{\ell-1}$ be such that
the $x_{j_h}$ are precisely the nodes of $T$ that have non-empty right subtrees. The \defterm{right interval partition}
of $T$ is the partition
\[
\set[\Big]{
\underbrace{\set{x_1,\ldots,x_{j_1}}}_{\text{$1$st right interval}},
\underbrace{\set{x_{j_1+1},\ldots,x_{j_2}}}_{\text{$2$nd right interval}},
\ldots,
\underbrace{\set{x_{j_{\ell-2}+1},\ldots,x_{j_{\ell-1}}}}_{\text{$\ell-1$-th right interval}},
\underbrace{\set{x_{\ell-1},\ldots,x_{m}}}_{\text{$\ell$-th right interval}}
};
\]
the $h$-th part of this partition is the \defterm{$h$-th right interval} of $T$. Consider the left and right intervals
of a binary tree with the same shape as the right-hand tree in \eqref{eq:bsteg}. Then the nodes lying in the left
intervals and right intervals of this tree are, respectively, those linked by solid lines in the left-hand and
right-hand diagrams below:
\begin{equation}
\begin{tikzpicture}[baseline=-7.5mm]
  \begin{scope}[tinybst]
    \node (root) {}
      child[sibling distance=16mm] { node (0) {}
        child { node (00) {}
          child { node (000) {} }
          child[missing]
        }
        child { node (01) {} }
      }
      child[sibling distance=16mm] { node (1) {}
        child { node (10) {}
          child { node (100) {} }
          child[missing]
        }
        child { node (11) {}
          child[missing]
          child { node (111) {} }
        }
      };
  \end{scope}
  \begin{scope}[very thick,line cap=round]
    \draw (111.center) -- (11.center) -- (1.center);
    \draw (100.center) -- (root.center);
    \draw (01.center) -- (0.center);
  \end{scope}
  \foreach\x in {000,00,01,0,100,10,111,11,1,root} {
    \draw[draw=black,fill=black] (\x.center) circle (.66mm);
  }
  \begin{scope}[every node/.style={font=\scriptsize,outer sep=1.5mm}]
    \node[anchor=east] at (000) {$1$st};
    \node[anchor=east] at (00) {$2$nd};
    \node[anchor=east] at (0) {$3$rd};
    \node[anchor=east] at (root) {$4$th};
    \node[anchor=north,xshift=3mm,yshift=.5mm] at (10) {$5$th};
    \node[anchor=west] at (1) {$6$th};
  \end{scope}
\end{tikzpicture}
\quad
\begin{tikzpicture}[baseline=-7.5mm]
  \begin{scope}[tinybst]
    \node (root) {}
      child[sibling distance=16mm] { node (0) {}
        child { node (00) {}
          child { node (000) {} }
          child[missing]
        }
        child { node (01) {} }
      }
      child[sibling distance=16mm] { node (1) {}
        child { node (10) {}
          child { node (100) {} }
          child[missing]
        }
        child { node (11) {}
          child[missing]
          child { node (111) {} }
        }
      };
  \end{scope}
  \begin{scope}[very thick,line cap=round]
    \draw (000.center) -- (00.center) -- (0.center);
    \draw (01.center) -- (root.center);
    \draw (100.center) -- (10.center) -- (1.center);
  \end{scope}
  \foreach\x in {000,00,01,0,100,10,111,11,1,root} {
    \draw[draw=black,fill=black] (\x.center) circle (.66mm);
  }
  \begin{scope}[every node/.style={font=\scriptsize,outer sep=1.5mm}]
    \node[anchor=east] at (000) {$1$st};
    \node[anchor=north,xshift=-3mm,yshift=.5mm] at (01) {$2$nd};
    \node[anchor=west] at (100) {$3$rd};
    \node[anchor=west] at (11) {$4$th};
    \node[anchor=west] at (111) {$5$th};
  \end{scope}
\end{tikzpicture}
\end{equation}
In a sense, the lengths of the left and right intervals of a binary seach tree are analogous to the lengths of the rows
in a quasi-ribbon tableau. However, the lengths of left and right intervals do not determine the shape of the tree, for the trees
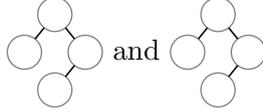
\[
\begin{tikzpicture}[tinybst,baseline=-6mm]
  \node (root) {}
    child { node (0) {} }
    child { node (1) {}
      child { node (10) {} }
      child[missing]
    };
\end{tikzpicture}
\text{ and }
\begin{tikzpicture}[tinybst,baseline=-6mm]
  \node (root) {}
  child { node (0) {} }
  child { node (1) {}
    child { node (10) {} }
    child[missing]
  };
\end{tikzpicture}
\]
both have left interval length $2,2$ and right interval lengths $1,2,1$:
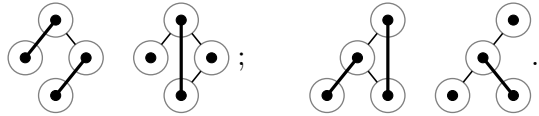
\[
\begin{tikzpicture}[tinybst,baseline=-6mm]
  \node (root) {}
    child { node (0) {} }
    child { node (1) {}
      child { node (10) {} }
      child[missing]
    };
  \begin{scope}[very thick,line cap=round]
    \draw (0.center) -- (root.center);
    \draw (10.center) -- (1.center);
  \end{scope}
  \foreach\x in {0,10,1,root} {
    \draw[draw=black,fill=black] (\x.center) circle (.66mm);
  }
\end{tikzpicture}
\quad
\begin{tikzpicture}[tinybst,baseline=-6mm]
  \node (root) {}
    child { node (0) {} }
    child { node (1) {}
      child { node (10) {} }
      child[missing]
    };
  \begin{scope}[very thick,line cap=round]
    \draw (10.center) -- (root.center);
  \end{scope}
  \foreach\x in {0,10,1,root} {
    \draw[draw=black,fill=black] (\x.center) circle (.66mm);
  }
\end{tikzpicture}
\;;\qquad
\begin{tikzpicture}[tinybst,baseline=-6mm]
  \node (root) {}
    child { node (0) {}
      child { node (00) {} }
      child { node (01) {} }
    }
    child[missing];
  \begin{scope}[very thick,line cap=round]
    \draw (00.center) -- (0.center);
    \draw (01.center) -- (root.center);
  \end{scope}
  \foreach\x in {00,01,0,root} {
    \draw[draw=black,fill=black] (\x.center) circle (.66mm);
  }
\end{tikzpicture}
\quad
\begin{tikzpicture}[tinybst,baseline=-6mm]
  \node (root) {}
    child { node (0) {}
      child { node (00) {} }
      child { node (01) {} }
    }
    child[missing];
  \begin{scope}[very thick,line cap=round]
    \draw (01.center) -- (0.center);
  \end{scope}
  \foreach\x in {00,01,0,root} {
    \draw[draw=black,fill=black] (\x.center) circle (.66mm);
  }
\end{tikzpicture}.
\]

The \defterm{canopy} of a binary tree $T$, denoted $\cnp{T}$ is the word over $\set{0,1}$ obtained by doing an infix
traversal of the nodes of $T$ and outputting a $1$ when an empty left subtree is encountered and $0$ when an empty right
subtree is encountered, then omitting the $1$ from the start and the $0$ from the end of the resulting word. (These
symbols correspond to the empty left subtree of the leftmost node and the empty right subtree of the rightmost node.)
For example, the canopies of the two binary trees in \eqref{eq:bsteg} are, respectively, $110101100$ and $001010011$:
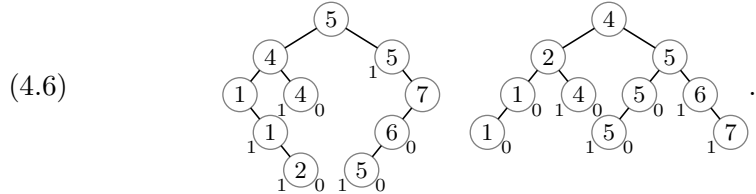
\begin{equation}
  \label{eq:twinbsteg}
  \begin{tikzpicture}[baseline=-10mm]
    \begin{scope}[tinybst]
      \node (root) {$5$}
      child[sibling distance=16mm] { node (0) {$4$}
        child { node (00) {$1$}
          child[missing]
          child { node (001) {$1$}
            child[missing]
            child { node (0011) {$2$} }
          }
        }
        child { node (01) {$4$} }
      }
      child[sibling distance=16mm] { node (1) {$5$}
        child[missing]
        child { node (11) {$7$}
          child { node (110) {$6$}
            child { node (1100) {$5$} }
            child[missing]
          }
          child[missing]
        }
      };
    \end{scope}
    \begin{scope}[every node/.style={font=\tiny}]
      \node at ($ (001) + (-2.5mm,-2mm) $) {$1$};
      \node at ($ (0011) + (-2.5mm,-2mm) $) {$1$};
      \node at ($ (0011) + (2.5mm,-2mm) $) {$0$};
      \node at ($ (01) + (-2.5mm,-2mm) $) {$1$};
      \node at ($ (01) + (2.5mm,-2mm) $) {$0$};
      \node at ($ (1) + (-2.5mm,-2mm) $) {$1$};
      \node at ($ (1100) + (-2.5mm,-2mm) $) {$1$};
      \node at ($ (1100) + (2.5mm,-2mm) $) {$0$};
      \node at ($ (110) + (2.5mm,-2mm) $) {$0$};
    \end{scope}
  \end{tikzpicture}
  \quad
  \begin{tikzpicture}[baseline=-10mm]
    \begin{scope}[tinybst]
      \node (root) {$4$}
      child[sibling distance=16mm] { node (0) {$2$}
        child { node (00) {$1$}
          child { node (000) {$1$} }
          child[missing]
        }
        child { node (01) {$4$} }
      }
      child[sibling distance=16mm] { node (1) {$5$}
        child { node (10) {$5$}
          child { node (100) {$5$} }
          child[missing]
        }
        child { node (11) {$6$}
          child[missing]
          child { node (111) {$7$} }
        }
      };
    \end{scope}
    \begin{scope}[every node/.style={font=\tiny}]
      \node at ($ (000) + (2.5mm,-2mm) $) {$0$};
      \node at ($ (00) + (2.5mm,-2mm) $) {$0$};
      \node at ($ (01) + (-2.5mm,-2mm) $) {$1$};
      \node at ($ (01) + (2.5mm,-2mm) $) {$0$};
      \node at ($ (100) + (-2.5mm,-2mm) $) {$1$};
      \node at ($ (100) + (2.5mm,-2mm) $) {$0$};
      \node at ($ (10) + (2.5mm,-2mm) $) {$0$};
      \node at ($ (11) + (-2.5mm,-2mm) $) {$1$};
      \node at ($ (111) + (-2.5mm,-2mm) $) {$1$};
    \end{scope}
  \end{tikzpicture}.
\end{equation}

Let $T_L$ and $T_R$ be binary trees. Then $(T_L,T_R)$ is a \defterm{pair of twin binary trees} if $\cnp{T_L}$ and
$\cnp{T_R}$ are complementary, in the sense that for all $i$, the $i$-th symbols of $\cnp{T_L}$ and $\cnp{T_R}$ are
unequal.

Now let $T_L$ be a left strict binary search tree and let $T_R$ be a right strict binary search tree. Then $(T_L,T_R)$
is a \defterm{pair of twin binary search trees} if $\inreading{T_L} = \inreading{T_R}$ and $\cnp{T_L}$ and $\cnp{T_R}$
are complementary (that is, the underlying binary trees form a pair of twin binary trees). Note that the binary search
trees in \eqref{eq:bsteg} form a pair of twin binary search trees, since they have the same infix reading (that is,
$1124455567$), and complementary canopies, as shown in \eqref{eq:twinbsteg}. The \defterm{shape} of a pair of twin
binary search trees $(T_L,T_R)$, denoted $\Sh{T_L,T_R}$, is simply the underlying pair of unlabelled rooted binary
trees.

The insertion algorithms for right (respectively, left) strict binary search trees adds the new symbol as a leaf node in the
unique place that maintains the property of being a right (respectively, left) strict binary search tree.

\begin{algorithm}[Left strict leaf insertion]
\label{alg:leftstrictinsertone}
~\par\nobreak
\textit{Input:} A left strict binary search tree $T$ and a symbol $a \in \aA$.

\textit{Output:} A left strict binary search tree $T \leftarrow a$.

\textit{Method:} If $T$ is empty, create a node and label it $a$. If $T$ is non-empty, examine the label $x$ of the root
node; if $a \geq x$, recursively insert $a$ into the right subtree of the root node; otherwise recursively insert $a$
into the left subtree of the root note. Output the resulting tree.
\end{algorithm}

\begin{algorithm}[Right strict leaf insertion]
\label{alg:rightstrictinsertone}
~\par\nobreak
\textit{Input:} A right strict binary search tree $T$ and a symbol $a \in \aA$.

\textit{Output:} A right strict binary search tree $a \rightarrow T$.

\textit{Method:} If $T$ is empty, create a node and label it $a$. If $T$ is non-empty, examine the label $x$ of the root
node; if $a \leq x$, recursively insert $a$ into the left subtree of the root node; otherwise recursively insert $a$
into the right subtree of the root note. Output the resulting tree.
\end{algorithm}

Using the leaf insertion algorithms described above, we can compute from a word in $\aA^*$ a left (respectively, right)
binary search tree.

\begin{algorithm}[Left strict insertion]
\label{alg:leftstrictinsert}
~\par\nobreak
\textit{Input:} A word $a_1\cdots a_k$, where $a_i \in \aA$.

\textit{Output:} A left strict binary search tree $\ltree{a_1\cdots a_k}$ and a standard increasing tree $\lrectree{a_1\cdots a_k}$.

\textit{Method:} Start with the empty binary search tree $T_0$ and an empty increasing tree $D_0$. For each
$i = 1,\ldots,k$, insert $a_{i}$ into $T_{i-1}$ as per \fullref{Algorithm}{alg:leftstrictinsertone}; let $T_i$ be the
resulting binary search tree. Build a increasing tree $D_i$ that has the same shape as $T_i$ by adding a node labelled
by $i$ to the tree $D_{i-1}$ in the same place as $a_i$ was inserted into $T_{i-1}$.

Output $T_k$ for $\ltree{a_1\cdots a_k}$ and $D_k$ for $\lrectree{a_1\cdots a_k}$.
\end{algorithm}

\begin{algorithm}[Right strict insertion]
\label{alg:rightstrictinsert}
~\par\nobreak
\textit{Input:} A word $a_1\cdots a_k$, where $a_i \in \aA$.

\textit{Output:} A right strict binary search tree $\rtree{a_1\cdots a_k}$ and a standard decreasing tree
$\rrectree{a_1\cdots a_k}$.

\textit{Method:} Start with the empty binary search tree $T_0$ and an empty decreasing tree $D_0$. For each
$i = 1,\ldots,k$, insert $a_{k-i+1}$ into $T_{i-1}$ as per \fullref{Algorithm}{alg:rightstrictinsertone}; let $T_i$ be
the resulting binary search tree. Build a decreasing tree $D_i$ that has the same shape as $T_i$ by adding a node
labelled by $k-i+1$ to the tree $D_{i-1}$ in the same place as $a_{k-i+1}$ was inserted into $T_{i-1}$.

Output $T_k$ for $\rtree{a_1\cdots a_k}$ and $D_k$ for $\rrectree{a_1\cdots a_k}$.
\end{algorithm}

For example, if $u = 5451761524$, then $\ltree{u}$ and $\rtree{u}$ are respectively the left- and right-hand trees in
\eqref{eq:bsteg}, and $\lrectree{u}$ and $\rrectree{u}$ are respectively the left- and right-hand trees in
\eqref{eq:incdectreeeg}. Note that \fullref{Algorithm}{alg:leftstrictinsert} proceeds through the word from left to
right, while \fullref{Algorithm}{alg:rightstrictinsert} proceeds through the word from right to left.

Let $u \in \aA^*$ be a word that contains no repeated symbols. The \defterm{decreasing tree of $u$}, denoted
$\dectree{u}$, is defined as follows: if $m$ is the maximum symbol that appears in $u$, so that $u = sms'$, then
$\dectree{u}$ has root labelled by $m$, and the left subtree of its root is $\dectree{s}$ and the right subtree of the
root is $\dectree{s'}$. That is,
\[
\dectree{u} = \dectree[\big]{sms'} =
\begin{tikzpicture}[smallbst,level distance=14mm,baseline=-5mm]
  \node (root) {$m$}
  child[sibling distance=15mm] { node[triangle,font=\scriptsize] (0) {$\dectree{s}$} }
  child[sibling distance=15mm] { node[triangle,font=\scriptsize] (1) {$\dectree{s'}$} };
\end{tikzpicture}.
\]

Similarly, the \defterm{increasing tree of $u$}, denoted $\inctree{u}$, is defined as follows: if $n$ is the minimum
symbol that appears in $u$, so that $u = tnt'$, then $\inctree{u}$ has root labelled by $n$, and the left subtree of its
root is $\inctree{t}$ and the right subtree of the root is $\inctree{t'}$. That is,
\[
\inctree{u} = \inctree[\big]{tnt'} =
\begin{tikzpicture}[smallbst,level distance=14mm,baseline=-5mm]
  \node (root) {$n$}
  child[sibling distance=15mm] { node[triangle,font=\scriptsize] (0) {$\inctree{t}$} }
  child[sibling distance=15mm] { node[triangle,font=\scriptsize] (1) {$\inctree{t'}$} };
\end{tikzpicture}
\]

Thus if $u$ is $4792\underline{10}13865$ (where the underline denotes a single symbol in $\aA$ that is not written using
a single decimal digit), then $\inctree{u}$ and $\dectree{u}$ are respectively the left- and right- hand trees in
\eqref{eq:incdectreeeg}.

\begin{remark}
  The infix reading of $\dectree{u}$ is $u$ and the infix reading of $\inctree{u}$ is $u$.
\end{remark}

\begin{proposition}[{\cite[Proposition~4.8]{giraudo_baxter2}}]
\label{prop:charrectrees}
For any word $u \in \aA^*$,
\begin{enumerate}
\item $\lrectree{u} = \inctree{\std{u}^{-1}}$;
\item $\rrectree{u} = \dectree{\std{u}^{-1}}$;
\end{enumerate}
and $\parens[\big]{\lrectree{u},\rrectree{u}}$ is a pair of twin binary trees.
\end{proposition}

(Note that since $\std{u}$ is a standard word, it can be viewed as a permutation in one-line notation, and thus its
inverse is defined.)

Before proceeding to define the sylvester and Baxter monoids, it is first necessary to prove two technical lemmata:

\begin{lemma}
  \label{lem:decinctreenewmaxshape}
  Let $\alpha,\alpha',\beta,\beta' \in \aA^*$ and $k \in \aA$ be such that $\abs{\alpha} = \abs{\alpha'}$,
  $\abs{\beta} = \abs{\beta'}$, the symbol $k$ is greater than every symbol in $\alpha,\alpha',\beta,\beta'$, and
  $\alpha k\beta$ and $\alpha' k\beta'$ are standard words.
  \begin{enumerate}
  \item If the decreasing trees $\dectree{\alpha\beta}$ and $\dectree{\alpha'\beta'}$ have the same shape as each other,
    then the decreasing trees $\dectree{\alpha k\beta}$ and $\dectree{\alpha'k\beta'}$ are of the same shape as each other.
  \item If the increasing trees $\inctree{\alpha\beta}$ and $\inctree{\alpha'\beta'}$ have the same shape as each other,
    then the increasing trees $\inctree{\alpha k\beta}$ and $\inctree{\alpha'k\beta'}$ are of the same shape as each other.
  \end{enumerate}
\end{lemma}

\begin{proof}
  \begin{enumerate}
    \item Suppose that the $\abs{\alpha}$-th symbol in $\alpha k\beta$ is $c$. Now consider two cases separately:
    \begin{enumerate}
    \item $c$ is a left child node in $\dectree{\alpha\beta}$. The infix reading of $\dectree{\alpha\beta}$ is
      $\alpha\beta$, so the left subtree of $c$ must have infix reading $\alpha_0$, where $\alpha = \alpha_0c$, and the right
      subtree of $c$ must have infix reading $\beta_0$, where $\beta_0$ is a prefix of $\beta$. Suppose
      $\beta = \beta_0\beta_1$; then $\beta_1$ is the infix reading of the part of the tree above and to the right of
      $c$. That is, $\dectree{\alpha\beta}$ must be as shown in the left-hand side of \eqref{eq:sylvshape1} below.
      \begin{equation}
        \label{eq:sylvshape1}
        \begin{tikzpicture}[baseline=-8mm]
          \begin{scope}[smallbst]
            \node[triangle] (roota) at (0,0) {$\beta_1$}
            child { node (a0) {$c$}
              child { node[triangle] (a00) {$\alpha_0$} }
              child { node[triangle] (a01) {$\beta_0$} }
            }
            child[missing];
          \end{scope}
          \begin{scope}[smallbst]
            \node (rootb) at (5,0) {$k$}
            child { node (b0) {$c$}
              child { node[triangle] (b00) {$\alpha_0$} }
              child[missing]
            }
            child  { node[triangle] (b1) {$\beta_1$}
              child { node[triangle] (b10) {$\beta_0$} }
              child[missing]
            };
          \end{scope}
          \node at ($ (roota) + (-5mm,6.5mm) $) {$\overbrace{\hbox{\vrule width 20mm height 0cm depth 0cm}}$};
          \node at ($ (rootb) + (0mm,6.5mm) $) {$\overbrace{\hbox{\vrule width 25mm height 0cm depth 0cm}}$};
          \node at ($ (roota) + (-5mm,10mm) $) {$\dectree{\alpha\beta}$};
          \node at ($ (rootb) + (0mm,10mm) $) {$\dectree{\alpha k\beta}$};
          \draw[mogrifyarrow] (1,-.8) -- (3,-.8);
        \end{tikzpicture}
      \end{equation}
      The tree on the right-hand side of \eqref{eq:sylvshape1} has infix reading $\alpha k\beta$. (Note that in the
      right-hand tree, the subtree $\beta_0$ attaches to the subtree $\beta_1$ in the same place as $c$ attaches to the
      subtree $\beta_1$ in the left-hand tree.) It is a decreasing tree since $k$ is the maximum symbol and the relative
      order of $c$, $\alpha_0$, $\beta_0$, and $\beta_1$ is correct since $\dectree{\alpha\beta}$ is a decreasing
      tree. Hence the right-hand tree in \eqref{eq:sylvshape1} is $\dectree{\alpha k\beta}$.

      Since $\dectree{\alpha'\beta'}$ has the same shape as $\dectree{\alpha\beta}$ and $\abs{\alpha} = \abs{\alpha'}$,
      the $\abs{\alpha}$-th symbol of $\alpha'\beta'$ is also a left child and so
      $\dectree{\alpha'\beta'}$ has the same shape as the left-hand tree in \eqref{eq:sylvshape1}. Since $k$ is also the
      maximum symbol in $\alpha'k\beta'$, the decreasing tree $\dectree{\alpha'k\beta'}$ has the same shape as the
      right-hand tree in \eqref{eq:sylvshape1}. Thus $\dectree{\alpha k\beta}$ and $\dectree{\alpha' k\beta'}$ have the
      same shape.

    \item $c$ is a right child node or the root node. The infix reading of $\dectree{\alpha\beta}$ is $\alpha\beta$, so
      the left subtree of $c$ must have infix reading $\alpha_1$, where $\alpha_1c$ is a suffix of $\alpha$, and the right
      subtree of $c$ must have infix reading $\beta$. Suppose $\alpha = \alpha_0\alpha_1c$; then $\alpha_0$ is the infix
      reading of the part of the tree above and to the left of $c$. That is, $\dectree{\alpha\beta}$ must be as shown in the
      left-hand side of \eqref{eq:sylvshape2} below.
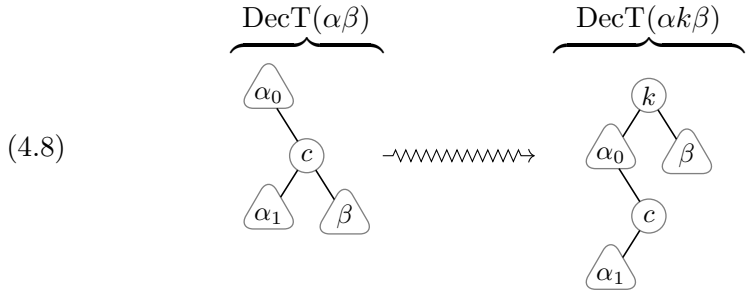
      \begin{equation}
        \label{eq:sylvshape2}
        \begin{tikzpicture}[baseline=-8mm]
          \begin{scope}[smallbst]
            \node[triangle] (roota) at (0,0) {$\alpha_0$}
            child[missing]
            child { node (a1) {$c$}
              child { node[triangle] (a10) {$\alpha_1$} }
              child { node[triangle] (a11) {$\beta$} }
            };
          \end{scope}
          \begin{scope}[smallbst]
            \node (rootb) at (5,0) {$k$}
            child  { node[triangle] (b0) {$\alpha_0$}
              child[missing]
              child { node (b01) {$c$}
                child { node[triangle] (b010) {$\alpha_1$} }
                child[missing]
              }
            }
            child { node[triangle] (b1) {$\beta$} };
          \end{scope}
          \node at ($ (roota) + (5mm,6.5mm) $) {$\overbrace{\hbox{\vrule width 20mm height 0cm depth 0cm}}$};
          \node at ($ (rootb) + (0mm,6.5mm) $) {$\overbrace{\hbox{\vrule width 25mm height 0cm depth 0cm}}$};
          \node at ($ (roota) + (5mm,10mm) $) {$\dectree{\alpha\beta}$};
          \node at ($ (rootb) + (0mm,10mm) $) {$\dectree{\alpha k\beta}$};
          \draw[mogrifyarrow] (1.5,-.8) -- (3.5,-.8);
        \end{tikzpicture}
      \end{equation}
      The tree on the right-hand side of \eqref{eq:sylvshape2} has infix reading $\alpha k\beta$. It is a decreasing tree
      since $k$ is the maximum symbol and the relative order of $c$, $\alpha_0$, $\alpha_1$, and $\beta$ is correct since
      $\dectree{\alpha\beta}$ is a decreasing tree. Hence the right-hand tree in \eqref{eq:sylvshape2} is
      $\dectree{\alpha k\beta}$.

      Since $\dectree{\alpha'\beta'}$ has the same shape as $\dectree{\alpha\beta}$ and $\abs{\alpha} = \abs{\alpha'}$,
      the $\abs{\alpha}$-th symbol of $\alpha'\beta'$ is also a right child or is also the root node and so $\dectree{\alpha'\beta'}$ has the same shape as
      the left-hand tree in \eqref{eq:sylvshape2}. Since $k$ is also the maximum symbol in $\alpha'k\beta'$, the decreasing
      tree $\dectree{\alpha'k\beta'}$ has the same shape as the right-hand tree in \eqref{eq:sylvshape2}. Thus
      $\dectree{\alpha k\beta}$ and $\dectree{\alpha' k\beta'}$ have the same shape.
    \end{enumerate}

  \item Since $k$ is greater than every symbol in $\alpha\beta$, the increasing tree $\inctree{\alpha k\beta}$ is
    obtained from $\inctree{\alpha\beta}$ by adding $k$ as a leaf node. With the same notation and similar reasoning as
    in case (1), one gets the following equivalents of \eqref{eq:sylvshape1} (when $c$ is a left child node) and
    \eqref{eq:sylvshape2} (when $c$ is a right child node or the root node):
      \begin{align*}
        &
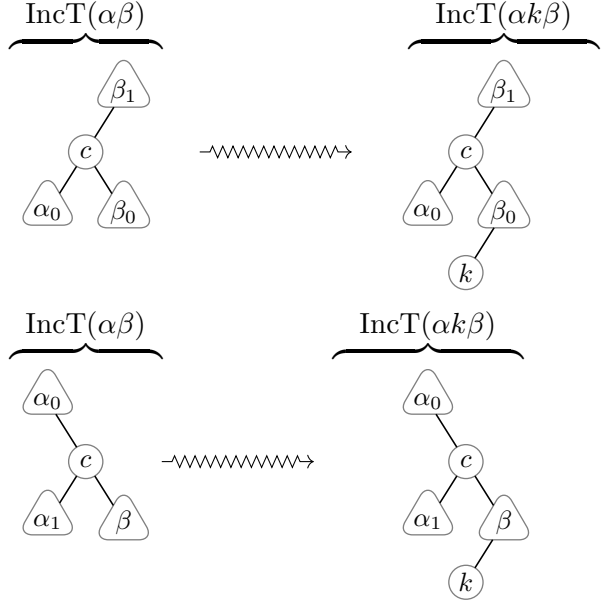
\begin{tikzpicture}[baseline=-8mm]
          \begin{scope}[smallbst]
            \node[triangle] (roota) at (0,0) {$\beta_1$}
            child { node (a0) {$c$}
              child { node[triangle] (a00) {$\alpha_0$} }
              child { node[triangle] (a01) {$\beta_0$} }
            }
            child[missing];
          \end{scope}
          \begin{scope}[smallbst]
            \node[triangle] (rootb) at (5,0) {$\beta_1$}
            child { node (b0) {$c$}
              child { node[triangle] (b00) {$\alpha_0$} }
              child { node[triangle] (b01) {$\beta_0$}
                child { node (b010) {$k$} }
                child[missing]
              }
            }
            child[missing];
          \end{scope}
          \node at ($ (roota) + (-5mm,6.5mm) $) {$\overbrace{\hbox{\vrule width 20mm height 0cm depth 0cm}}$};
          \node at ($ (rootb) + (0mm,6.5mm) $) {$\overbrace{\hbox{\vrule width 25mm height 0cm depth 0cm}}$};
          \node at ($ (roota) + (-5mm,10mm) $) {$\inctree{\alpha\beta}$};
          \node at ($ (rootb) + (0mm,10mm) $) {$\inctree{\alpha k\beta}$};
          \draw[mogrifyarrow] (1,-.8) -- (3,-.8);
        \end{tikzpicture}
        \\
        &\begin{tikzpicture}[baseline=-8mm]
          \begin{scope}[smallbst]
            \node[triangle] (roota) at (0,0) {$\alpha_0$}
            child[missing]
            child { node (a1) {$c$}
              child { node[triangle] (a10) {$\alpha_1$} }
              child { node[triangle] (a11) {$\beta$} }
            };
          \end{scope}
          \begin{scope}[smallbst]
            \node[triangle] (rootb) at (5,0) {$\alpha_0$}
            child[missing]
            child { node (b1) {$c$}
              child { node[triangle] (b10) {$\alpha_1$} }
              child { node[triangle] (b11) {$\beta$}
                child { node (b110) {$k$} }
                child[missing]
              }
            };
          \end{scope}
          \node at ($ (roota) + (5mm,6.5mm) $) {$\overbrace{\hbox{\vrule width 20mm height 0cm depth 0cm}}$};
          \node at ($ (rootb) + (0mm,6.5mm) $) {$\overbrace{\hbox{\vrule width 25mm height 0cm depth 0cm}}$};
          \node at ($ (roota) + (5mm,10mm) $) {$\inctree{\alpha\beta}$};
          \node at ($ (rootb) + (0mm,10mm) $) {$\inctree{\alpha k\beta}$};
          \draw[mogrifyarrow] (1.5,-.8) -- (3.5,-.8);
        \end{tikzpicture}
      \end{align*}
      In each case, the right-hand tree is an increasing tree with infix read $\alpha k\beta$, by the fact that $k$ is
      the maximum symbol and the relative order of $c$, $\alpha_0$, $\alpha_1$, and $\beta$ is correct since
      $\inctree{\alpha\beta}$ is an increasing tree. Hence in both cases the right-hand tree is
      $\inctree{\alpha k\beta}$. Thus, by similar reasoning to case~(1), $\inctree{\alpha k\beta}$ and
      $\inctree{\alpha' k\beta'}$ have the same shape.
  \end{enumerate}
\end{proof}

The second technical lemma is simply the dual of the first:

\begin{lemma}
  \label{lem:decinctreenewminshape}
  Let $\alpha,\alpha',\beta,\beta' \in \aA^*$ be such that $\abs{\alpha} = \abs{\alpha'}$,
  $\abs{\beta} = \abs{\beta'}$, and
  $\alpha 1\beta$ and $\alpha' 1\beta'$ are standard words.
  \begin{enumerate}
  \item If the decreasing trees $\dectree{\alpha\beta}$ and $\dectree{\alpha'\beta'}$ have the same shape as each other,
    then the decreasing trees $\dectree{\alpha 1\beta}$ and $\dectree{\alpha'1\beta'}$ are of the same shape as each other.
  \item If the increasing trees $\inctree{\alpha\beta}$ and $\inctree{\alpha'\beta'}$ have the same shape as each other,
    then the increasing trees $\inctree{\alpha 1\beta}$ and $\inctree{\alpha'1\beta'}$ are of the same shape as each other.
  \end{enumerate}
\end{lemma}

Equipped with these notions of binary search trees and insertion algorithms, it is now possible to use them to define
the sylvester and Baxter monoids (approach~2 in the introduction).

\subsection{Sylvester monoid}

Define the \defterm{sylvester $\plit$-symbol} of $u \in \aA^*$ by $\psylv{u} = \rtree{u}$, and the \defterm{sylvester
  $\qlit$-symbol} of $u \in \aA^*$ by $\qsylv{u} = \rrectree{u}$. The definition of the relation $\sylvcong$ using right
strict binary search trees and insertion is:
\[
u \sylvcong v \iff \psylv{u} = \psylv{v}.
\]
Using this definition, it can be shown that $\sylvcong$ is a congruence on $\aA^*$, which is known as the
\defterm{sylvester congruence}. The factor monoid $\aA^*/{\sylvcong}$ is the \defterm{sylvester monoid} and
is denoted $\sylv$. The congruence $\sylvcong$ naturally restricts to a congruence on $\aA_n^*$, and the factor monoid
$\aA_n^*/\sylvcong$ is the \defterm{sylvester monoid of rank $n$} and is denoted $\sylv_n$.

The monoid $\sylv$ is presented by $\pres{\aA}{\drel{R}_\sylv}$, where
\[
\drel{R}_\sylv = \gset[\big]{(cavb,acvb)}{a \leq b < c, v \in \aA^*};
\]
the monoid $\sylv_n$ is presented by $\pres{\aA_n}{\drel{R}_\sylv}$, where the set of defining relations
$\drel{R}_\sylv$ is naturally restricted to $\aA_n^*\times \aA_n^*$. Notice that $\sylv$ and $\sylv_n$ are
multihomogeneous. (See \cite{cgm_homogeneous} for a general discussion of homogeneous and multihomogeneous monoids.)

It is straightforward to see that the map $u \mapsto \parens[\big]{\psylv{u},\qsylv{u}}$ is a bijection between words
in $\aA^*$ and pairs consisting of a right strict binary search tree and a decreasing tree of the same shape; this is
the sylvester analogue of the Robinson--Schensted correspondence. For instance, if $\psylv{u}$ and $\qsylv{u}$ are the
right-hand trees in \eqref{eq:bsteg} and \eqref{eq:incdectreeeg} respectively, then $u = 5451761524$.

For any word $u \in \aA^*$, the set of words in the $\sylvcong$-class of $u$ is the \defterm{sylvester class} of
$u$. Extend this terminology to right strict binary search trees: for such a tree $T$, the set of words $w \in \aA^*$
such that $\psylv{w} = T$ is the \defterm{sylvester class} of $T$. The left-to-right postfix reading $\postreading{T}$
is an element of the sylvester class of $T$: that is, $\psylv{\postreading{T}} = T$. (Note that the infix reading
$\inreading{T}$ is \emph{not} in general an element of the sylvester class of $T$.) The following lemma follows
immediately from the definition of the left-to-right postfix reading:

\begin{lemma}
\label{lem:sylvcanonicalrobinsonschensted}
Let $T$ be an unlabelled binary tree. Under the sylvester analogue of the Robinson--Schensted correspondence, the set of
postfix readings of right strict binary search trees of shape $T$ is obtained by fixing a particular sylvester
$\qlit$-symbol of shape $T$ and varying the sylvester $\plit$-symbol of over all right strict binary search trees of
shape $T$.
\end{lemma}

The following corollary of \fullref{Proposition}{prop:charrectrees} does not seem to have been stated before:

\begin{corollary}
\label{corol:charsylvcong}
For $u,v \in \aA^*$,
\begin{align*}
u \sylvcong v \iff{}& \wt{u} = \wt{v} \\
&\land \Sh[\big]{\dectree{\std{u}^{-1}}} = \Sh[\big]{\dectree{\std{v}^{-1}}}.
\end{align*}
\end{corollary}

\begin{proof}
  Suppose that $u \sylvcong v$. Then $\rtree{u} = \rtree{v}$. Thus
  $\wt{u} = \wt{\rtree{u}} = \wt{\rtree{v}} = \wt{v}$ and, by \fullref{Proposition}{prop:charrectrees},
  \begin{multline*}
    \Sh[\big]{\dectree{\std{u}^{-1}}} = \Sh{\qsylv{u}} = \Sh{\psylv{u}} \\
    = \Sh{\psylv{v}} = \Sh{\qsylv{v}} = \Sh[\big]{\dectree{\std{v}^{-1}}}.
  \end{multline*}

  On the other hand, suppose $\wt{u} = \wt{v}$ and
  $\Sh[\big]{\dectree{\std{u}^{-1}}} = \Sh[\big]{\dectree{\std{v}^{-1}}}$. Then
  \begin{multline*}
    \Sh{\psylv{v}} = \Sh{\qsylv{v}} = \Sh[\big]{\dectree{\std{v}^{-1}}} \\
    = \Sh[\big]{\dectree{\std{u}^{-1}}} = \Sh{\qsylv{u}} = \Sh{\psylv{u}}.
  \end{multline*}
  Thus $\psylv{u}$ and $\psylv{v}$ have the same shape. Since $\wt{u} = \wt{v}$, the trees $\psylv{u}$ and $\psylv{v}$
  have the same content. By \fullref{Proposition}{prop:infixreading}, the infix reading of a binary search tree is the
  unique weakly increasing word with the same content, and so it follows that the infix readings of $\psylv{u}$ and
  $\psylv{v}$ are equal. Since $\psylv{u}$ and $\psylv{v}$ have the same shape, they are thus identical. Hence
  $u \sylvcong v$.
\end{proof}

\subsection{Baxter monoid}

The definition of the relation $\baxtcong$ uses pairs of twin binary search trees. For any word $u \in \aA^*$, the
binary search trees $\ltree{u}$ and $\rtree{u}$ form a pair of twin binary search trees. Define the \defterm{Baxter
  $\qlit$-symbol} of $u$ by $\pbaxt{u} = \parens[\big]{\ltree{u},\rtree{u}}$, and the \defterm{Baxter
  $\qlit$-symbol} of $u$ by $\qbaxt{u} = \parens[\big]{\lrectree{u},\rrectree{u}}$. Define $\baxtcong$ by:
\[
u \baxtcong v \iff \pbaxt{u} = \pbaxt{v}.
\]
Using this definition, it follows that $\baxtcong$ is a congruence on $\aA^*$, which is known as the \defterm{Baxter
  congruence}. The factor monoid $\aA^*/{\baxtcong}$ is the \defterm{Baxter monoid} and is denoted $\baxt$. The
congruence $\baxtcong$ naturally restricts to a congruence on $\aA_n^*$, and the factor monoid $\aA_n^*/\baxtcong$ is
the \defterm{Baxter monoid of rank $n$} and is denoted $\baxt_n$.

The monoid $\baxt$ is presented by $\pres{\aA}{\drel{R}_\baxt}$, where
\begin{align*}
\drel{R}_\baxt ={}&\gset[\big]{(cudavb,cuadvb)}{a \leq b < c \leq d, u,v \in \aA^*} \\
&\cup \gset{(budavc,buadvc)}{a < b \leq c < d, u,v \in \aA^*};
\end{align*}
the monoid $\baxt_n$ is presented by $\pres{\aA_n}{\drel{R}_\baxt}$, where the set of defining relations
$\drel{R}_\baxt$ is naturally restricted to $\aA_n^* \times \aA_n^*$. Note that $\baxt$ and $\baxt_n$ are
multihomogeneous.

The map
\[
u \mapsto \parens[\big]{\pbaxt{u},\qbaxt{u}} = \parens[\Big]{\parens[\big]{\ltree{u},\rtree{u}},\parens[\big]{\lrectree{u},\rrectree{u}}}
\]
is a bijection between words in $\aA^*$ and pairs consisting of a pair of twin binary search tree and a pair made up of
an increasing and a decreasing tree of the same shape as the corresponding binary search trees; this is the Baxter
analogue of the Robinson--Schensted correspondence. For instance, if $\pbaxt{u}$ is the pair of twin binary search trees
in \eqref{eq:bsteg}, and $\qbaxt{u}$ is the pair of increasing and decreasing trees in \eqref{eq:incdectreeeg}, then $u = 5451761524$.

For example, if $u = 5451761524$, then $\std{u} = 6471\underline{10}92835$, so
$(\std{u})^{-1} = 4792\underline{10}13865$. Hence $\lrectree{u}$ and $\inctree{(\std{l})^{-1}}$ are both the left-hand
tree in \eqref{eq:incdectreeeg}, and $\rrectree{u}$ and $\dectree{(\std{u})^{-1}}$ are both the right-hand tree in
\eqref{eq:incdectreeeg}.

For any word $u \in \aA^*$, the set of words in the $\baxtcong$-class of $u$ is the \defterm{Baxter class} of
$u$. Extend this terminology to pairs of twin binary search trees: for such a pair $(T_L,T_R)$, the set of words
$w \in \aA^*$ such that $\pbaxt{w} = T$ is the \defterm{Baxter class} of $T$.

There is a straightforward method for computing words in the Baxter class of a pair of twin binary search trees
$(T_L,T_R)$:

\begin{method}
\label{method:baxterreading}

\textit{Input:} A pair of twin binary search trees $(T_L,T_R)$.

\textit{Output:} A word in the Baxter class of $(T_L,T_R)$.
\begin{enumerate}
\item Set $(U_L,U_R)$ to be $(T_L,T_R)$. (Throughout this computation, $U_L$ is a forest of
  left strict binary search trees and $U_R$ is a right strict binary search tree.)
\item If $U_L$ and $U_R$ are empty, halt.
\item Given some $(U_L,U_R)$, choose and output some symbol $a$ that labels a root of some tree in the forest $U_L$ and a leaf
  of the tree $U_R$.
\item Delete the corresponding root vertex of $U_L$ and the corresponding leaf vertex of $U_L$.
\item Go to step~2.
\end{enumerate}
\end{method}

This is essentially \cite[Algorithm on p.133]{giraudo_baxter2}, except that the method given here is non-deterministic
in that there may be several choices for $a$ in step~3. As these choices vary, all words in the Baxter class of
$(T_L,T_R)$ are obtained. If, in step~3, one always chooses the leftmost possible $a$, call the resulting word the
\defterm{left-consistent reading} of $(T_L,T_R)$. (By \cite[Proposition~4.16]{giraudo_baxter2}, the left-consistent
reading of $(T_L,T_R)$ is the lexicographically smallest word in the corresponding Baxter class.) The following lemma
follows immediately from the definition of the left-consistent reading:

\begin{lemma}
\label{lem:canonicalbaxtrobinsonschensted}
Let $(T_L,T_R)$ be a pair of twin binary trees. Under the Baxter analogue of the Robinson--Schensted correspondence, the
set of left-consistent readings of pairs of twin binary search trees of shape $(T_L,T_R)$ is obtained by fixing a
particular Baxter $\qlit$-symbol of shape $(T_L,T_R)$ and varying the Baxter $\plit$-symbol over all pairs of twin
binary search trees of shape $(T_L,T_R)$.
\end{lemma}

The following result is the analogy of \fullref{Corollary}{corol:charsylvcong} for the Baxter monoid:

\begin{corollary}
\label{corol:charbaxtcong}
For $u,v \in \aA^*$,
\begin{align*}
u \baxtcong v \iff{}& \wt{u} = \wt{v} \\
  &\land \Sh[\big]{\dectree{\std{u}^{-1}}} = \Sh[\big]{\dectree{\std{v}^{-1}}} \\
  &\land \Sh[\big]{\inctree{\std{u}^{-1}}} = \Sh[\big]{\inctree{\std{v}^{-1}}}.
\end{align*}
\end{corollary}

\begin{proof}
  Follow the reasoning in the proof of \fullref{Corollary}{corol:charsylvcong}, but working additionally with
  $\ltree{\cdot}$, $\lrectree{\cdot}$, and $\inctree{\std{\cdot}^{-1}}$.
\end{proof}


\section{Quasi-crystals with abstract shapes}
\label{sec:abstractshapes}

An \defterm{abstract shape} is a map $\shapelit : \aA^* \to S$ satisfying the following axioms:
\begin{itemize}
\item[S1] The map $\shapelit$ is invariant under standardization, in the sense that for all $u \in \aA^*$
  $\shape{u} = \shape{\std{u}}$.
\item[S2] For all $u,v \in \aA^*$ and $a \in \aA$, if $\wt{u} = \wt{v}$ and $\shape{u} = \shape{v}$, then
  $\shape{ua} = \shape{va}$ and $\shape{au} = \shape{av}$ for all $w \in \aA^*$.
\end{itemize}

Explicit examples of shapes will be given in the next section. For the moment, note that by
\fullref{Proposition}{prop:eipreservesstd} and \fullref{Corollary}{corol:stdinsamecomp}, axiom~S1 is equivalent to
saying that $\shapelit$ is invariant under the quasi-Kashiwara operators, in the sense that for all $u \in \aA^*$ and
all $i \in \nset$, if $\e_i(u)$ is defined, then $\shape{u} = \shape{\e_i(u)}$, and if $\f_i(u)$ is defined,
$\shape{u} = \shape{\f_i(u)}$. Thus, if axiom~S1 is satisfied, every element of a given connected component of
$\Gamma(\hypo)$ will have the same image under $\shapelit$.

For $u,v \in \aA^*$, a quasi-crystal isomorphism $\theta : \Gamma(\hypo,u) \to \Gamma(\hypo,v)$ is
\defterm{shape-preserving} (with respect to an abstract shape $\shapelit$) if $\shape{w} = \shape{\theta{w}}$ for all
$w \in \Gamma(\hypo,u)$.

Define a relation $\shapisom$ on the free monoid $\aA^*$ as follows: $u \shapisom v$ if and only if there is a
shape-preserving quasi-crystal isomorphism $\theta : \Gamma(\hypo,u) \to \Gamma(\hypo,v)$ such that $\theta(u) = v$.
That is, $u \shapisom v$ if and only if $u$ and $v$ have the same image under $\shapelit$ and are in the same position
in isomorphic connected components of $\Gamma(\hypo)$. Note that ${\shapisom} \subseteq {\hypoisom}$.

\begin{proposition}
  \label{prop:isomdefinescong}
  The relation $\shapisom$ is a congruence on the free monoid $\aA^*$.
\end{proposition}

\begin{proof}
  It is clear from the definition that $\shapisom$ is an equivalence relation; it thus remains to prove that $\shapisom$ is
  compatible with multiplication in $\aA^*$.

  Suppose $u \shapisom v$ and $w \in \aA^*$. Suppose $w = w_1\cdots w_k$ (where $w_i \in \aA$).  The aim is to proceed
  by induction on $i$ and show there is a shape-preserving quasi-crystal isomorphism
  $\theta_i : \Gamma(\hypo,uw_1\cdots w_i) \to \Gamma(\hypo,vw_1\cdots w_i)$ such that $\theta(uw_1\cdots w_i) = vw_1\cdots w_i$ for all
  $i \in \set{0,1,\ldots,k}$ (formally taking $w_1\cdots w_i$ to be $\emptyword$ when $i=0$).

  For $i=0$, since $u \shapisom v$, there is by definition a shape-preserving quasi-crystal isomorphisms
  $\theta_0 : \Gamma(\hypo,u) \to \Gamma(\hypo,v)$ such that $\theta(u) = v$. This is the basis of the induction.

  So assume there is a shape-preserving quasi-crystal isomorphism
  $\theta_i : \Gamma(\hypo,uw_1\cdots w_i) \to \Gamma(\hypo,vw_1\cdots w_i)$ such that
  $\theta_i(uw_1\cdots w_i) = vw_1\cdots w_i$. In particular, $uw_1\cdots w_i \hypoisom vw_1\cdots w_i$. Since $\hypoisom$
  is a congruence, $uw_1\cdots w_iw_{i+1} \hypoisom vw_1\cdots w_iw_{i+1}$, and so there is a quasi-crystal isomorphism
  (which is not yet known to be shape-preserving)
  $\theta_{i+1} : \Gamma(\hypo,uw_1\cdots w_{i+1}) \to \Gamma(\hypo,vw_1\cdots w_{i+1})$ such that
  $\theta_{i+1}(uw_1\cdots w_{i+1}) = vw_1\cdots w_{i+1}$. Furthermore, since $\theta_i$ is shape-preserving,
  $\shape{uw_1\cdots w_i} = \shape{\theta(uw_1\cdots w_i)} = \shape{vw_1\cdots w_i}$. Combine this with
  $uw_1\cdots w_i \hypoisom vw_1\cdots w_i$ and use axiom~S2 to see that
  $\shape{uw_1\cdots w_{i+1}} = \shape{vw_1\cdots w_{i+1}}$. Thus, by axiom S1, every element in the connected
  components $\Gamma(\hypo,uw_1\cdots w_{i+1})$ and $\Gamma(\hypo,vw_1\cdots w_{i+1})$ have the same shape. Hence
  $\theta_{i+1}$ is shape-preserving.

  Hence, by induction, $\theta_k : \Gamma(\hypo,uw) \to \Gamma(\hypo,vw)$ is a shape-preserving quasi-crystal
  isomorphism; therefore $uw \shapisom vw$.

  Similar reasoning shows that $wu \shapisom wv$; hence $\shapisom$ is compatible with multiplication and is thus a
  congruence.
\end{proof}

Thus every abstract shape $\shapelit$ gives rise to a monoid $\aA^*/{\shapisom}$.

\section{Crystallizing the sylvester monoid}
\label{sec:sylvcrystal}

\subsection{Definition by crystals}

Let $S$ be the set of unlabelled binary trees and define $\sylvshape{u} = \Sh{\dectree{\std{u}^{-1}}}$ for all
$u \in \aA^*$.

\begin{proposition}
  \label{prop:sylvshape}
  $\sylvshapelit$ is an abstract shape.
\end{proposition}

\begin{proof}
  By \fullref{Proposition}{prop:eipreservesstd}, the map $\sylvshapelit$ is invariant under the quasi-Kashiwara operators,
  and so satisfies axiom~S1.

  To see that $\sylvshapelit$ satisfies~S2, proceed as follows. Let $u,v \in \aA^*$ be such that $u \sim v$ and
  $\sylvshape{u} = \sylvshape{v}$, and let $a \in \aA$.

  The first aim is to prove that $\sylvshape{ua} = \sylvshape{va}$. When one computes $\std{ua}$ and $\std{va}$, the
  symbol $a$ is replaced by the same symbol $b$ in both words, since $u$ and $v$ have the same content. Let
  $k = \wlen{u}+1$. In $\std{ua}$ (viewed as a permutation), $k$ is mapped to $b$. Thus, $k$ is the $b$-th symbol, and
  the maximum symbol, in both $\std{ua}^{-1}$ and $\std{va}^{-1}$. That is, $\std{u}^{-1} = \alpha\beta$,
  $\std{ua}^{-1} = \alpha k\beta$, $\std{v}^{-1} = \alpha'\beta'$ and $\std{va}^{-1} = \alpha' k\beta'$ for some
  $\alpha,\alpha',\beta,\beta' \in \aA^*$, such that $\abs{\alpha} = \abs{\alpha'}$, $\abs{\beta} = \abs{\beta'}$, the
  symbol $k$ is greater than every symbol in $\alpha,\alpha',\beta,\beta'$, and $\alpha k\beta$ and $\alpha' k\beta'$
  are standard words. Furthermore, $\dectree{\std{u}^{-1}} = \dectree{\alpha\beta}$ and
  $\dectree{\std{v}^{-1}} = \dectree{\alpha'\beta'}$ have the same shape, since $\sylvshape{u} = \sylvshape{v}$. Thus,
  by \fullref[(1)]{Lemma}{lem:decinctreenewmaxshape}, $\dectree{\std{ua}^{-1}} = \dectree{\alpha k\beta}$ and
  $\dectree{\std{va}^{-1}} = \dectree{\alpha'k\beta'}$ have the same shape. Hence $\sylvshape{ua} = \sylvshape{va}$.

  The next aim is to prove that $\sylvshape{au} = \sylvshape{av}$. When one computes $\std{au}$ and $\std{av}$, the
  symbol $a$ is replaced by the same symbol $b$ in both words, since $u$ and $v$ have the same content. In $\std{au}$
  (viewed as a permutation), $1$ is mapped to $b$. Thus, $1$ is the $b$-th symbol, and the minimum symbol, in both
  $\std{au}^{-1}$ and $\std{av}^{-1}$. By similar reasoning to the above, but using
  \fullref[(1)]{Lemma}{lem:decinctreenewminshape}, one sees that $\dectree{\std{au}^{-1}}$ and $\dectree{\std{av}^{-1}}$
  have the same shape. Hence $\sylvshape{au} = \sylvshape{av}$.
\end{proof}

As a consequence of \fullref{Proposition}{prop:sylvshape}, one can define a relation $\sylvisom$ on the free monoid
$\aA^*$ to be the relation $\shapisom$ from \fullref{Section}{sec:abstractshapes} with the abstract shape
$\shapelit = \sylvshapelit$. By \fullref{Propositions}{prop:isomdefinescong} and \ref{prop:sylvshape}, $\sylvisom$ is a
congruence.

When viewing the quasi-crystal graph $\Gamma(\hypo)$ in the context of $\sylvshapelit$-preserving isomorphisms, denote it
by $\Gamma(\sylv)$. Similarly, denote the connected component $\Gamma(\hypo,u)$ by $\Gamma(\sylv,u)$. (This is arguably
an abuse of notation, since $\Gamma(\hypo)$ and $\Gamma(\sylv)$ are the same graph, and the difference is in the notion
of isomorphism and the relation to which it gives rise.)

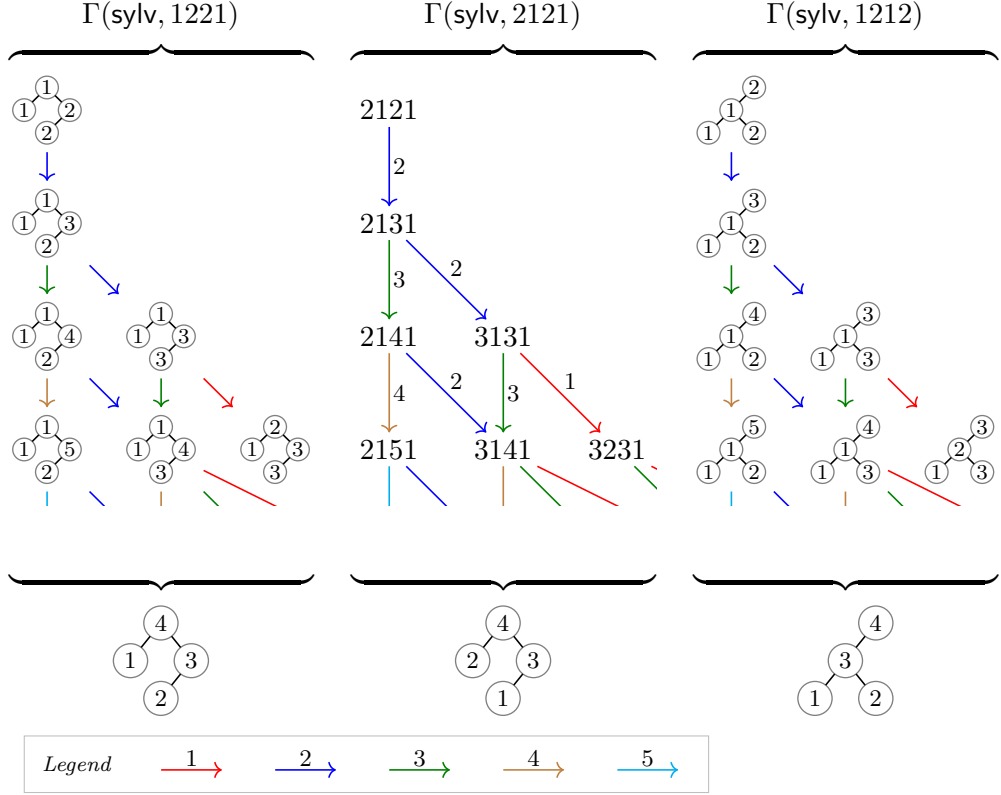
\begin{figure}[t]
  \centering
  \begin{tikzpicture}[bigcrystal,x=15mm,y=15mm,labelledcolouredcrystaledges]
    \begin{scope}[xshift=0mm]
      \begin{scope}
        \node (1221) at (0,0) {$\tikz[microbst]
          \node {$1$}
          child { node {$1$} }
          child { node {$2$}
            child { node {$2$} }
            child[missing]
          };$};
        \node (1231) at (0,-1) {$\tikz[microbst]
          \node {$1$}
          child { node {$1$} }
          child { node {$3$}
            child { node {$2$} }
            child[missing]
          };$};
        \node (1241) at (0,-2) {$\tikz[microbst]
          \node {$1$}
          child { node {$1$} }
          child { node {$4$}
            child { node {$2$} }
            child[missing]
          };$};
        \node (1331) at (1,-2) {$\tikz[microbst]
          \node {$1$}
          child { node {$1$} }
          child { node {$3$}
            child { node {$3$} }
            child[missing]
          };$};
        \node (1251) at (0,-3) {$\tikz[microbst]
          \node {$1$}
          child { node {$1$} }
          child { node {$5$}
            child { node {$2$} }
            child[missing]
          };$};
        \node (1341) at (1,-3) {$\tikz[microbst]
          \node {$1$}
          child { node {$1$} }
          child { node {$4$}
            child { node {$3$} }
            child[missing]
          };$};
        \node (1332) at (2,-3) {$\tikz[microbst]
          \node {$2$}
          child { node {$1$} }
          child { node {$3$}
            child { node {$3$} }
            child[missing]
          };$};
        \coordinate (1261) at (0,-4) {};
        \coordinate (1351) at (1,-4) {};
        \coordinate (1342) at (3,-4) {};
        \coordinate (1441) at (2,-4) {};
        \coordinate (2332) at (4,-4) {};
      \end{scope}
      \begin{scope}[colouredcrystaledges]
        \draw[f2] (1221) to (1231);
        \draw[f2] (1231) to (1331);
        \draw[f3] (1231) to (1241);
        \draw[f4] (1241) to (1251);
        \draw[f2] (1241) to (1341);
        \draw[f3] (1331) to (1341);
        \draw[f1] (1331) to (1332);
        \draw[f5] (1251) to (1261);
        \draw[f2] (1251) to (1351);
        \draw[f4] (1341) to (1351);
        \draw[f1] (1341) to (1342);
        \draw[f3] (1341) to (1441);
        \draw[f1] (1332) to (2332);
        \draw[f3] (1332) to (1342);
      \end{scope}
      \fill[color=white,path fading=fade down] ($ (1251) + (-.2,-.5) $) rectangle ($ (1332) + (2,-.8) $);
      \fill[color=white] ($ (1251) + (-.2,-.79) $) rectangle ($ (1332) + (2,-1.1) $);
      \fill[color=white,path fading=fade right] ($ (1332) + (.35,0) $) rectangle ($ (1332) + (.6,-1) $);
      \fill[color=white] ($ (1332) + (.58,0) $) rectangle ($ (1332) + (2,-1) $);
    \end{scope}
    \begin{scope}[xshift=45mm]
      \begin{scope}
        \node (2121) at (0,-0) {$2121$};
        \node (2131) at (0,-1) {$2131$};
        \node (2141) at (0,-2) {$2141$};
        \node (3131) at (1,-2) {$3131$};
        \node (2151) at (0,-3) {$2151$};
        \node (3141) at (1,-3) {$3141$};
        \node (3231) at (2,-3) {$3231$};
        \coordinate (2161) at (0,-4) {};
        \coordinate (3151) at (1,-4) {};
        \coordinate (3241) at (3,-4) {};
        \coordinate (4141) at (2,-4) {};
        \coordinate (3232) at (4,-4) {};
        \draw[f2] (2121) to (2131);
        \draw[f2] (2131) to (3131);
        \draw[f3] (2131) to (2141);
        \draw[f4] (2141) to (2151);
        \draw[f2] (2141) to (3141);
        \draw[f3] (3131) to (3141);
        \draw[f1] (3131) to (3231);
      \end{scope}
      \begin{scope}[colouredcrystaledges]
        \draw[f5] (2151) to (2161);
        \draw[f2] (2151) to (3151);
        \draw[f4] (3141) to (3151);
        \draw[f1] (3141) to (3241);
        \draw[f3] (3141) to (4141);
        \draw[f1] (3231) to (3232);
        \draw[f3] (3231) to (3241);
      \end{scope}
      \fill[color=white,path fading=fade down] ($ (2151) + (-.2,-.5) $) rectangle ($ (3231) + (2,-.8) $);
      \fill[color=white] ($ (2151) + (-.2,-.79) $) rectangle ($ (3231) + (2,-1.1) $);
      \fill[color=white,path fading=fade right] ($ (3231) + (.35,0) $) rectangle ($ (3231) + (.6,-1) $);
      \fill[color=white] ($ (3231) + (.58,0) $) rectangle ($ (3231) + (2,-1) $);
    \end{scope}
    \begin{scope}[xshift=90mm]
      \begin{scope}
        \node (1212) {$\tikz[microbst]
          \node {$2$}
          child { node {$1$}
            child { node {$1$} }
            child { node {$2$} }
          }
          child[missing];$};
        \node (1213) at (0,-1) {$\tikz[microbst]
          \node {$3$}
          child { node {$1$}
            child { node {$1$} }
            child { node {$2$} }
          }
          child[missing];$};
        \node (1214) at (0,-2) {$\tikz[microbst]
          \node {$4$}
          child { node {$1$}
            child { node {$1$} }
            child { node {$2$} }
          }
          child[missing];$};
        \node (1313) at (1,-2) {$\tikz[microbst]
          \node {$3$}
          child { node {$1$}
            child { node {$1$} }
            child { node {$3$} }
          }
          child[missing];$};
        \node (1215) at (0,-3) {$\tikz[microbst]
          \node {$5$}
          child { node {$1$}
            child { node {$1$} }
            child { node {$2$} }
          }
          child[missing];$};
        \node (1314) at (1,-3) {$\tikz[microbst]
          \node {$4$}
          child { node {$1$}
            child { node {$1$} }
            child { node {$3$} }
          }
          child[missing];$};
        \node (1323) at (2,-3) {$\tikz[microbst]
          \node {$3$}
          child { node {$2$}
            child { node {$1$} }
            child { node {$3$} }
          }
          child[missing];$};
        \coordinate (1216) at (0,-4) {};
        \coordinate (1315) at (1,-4) {};
        \coordinate (1324) at (3,-4) {};
        \coordinate (1414) at (2,-4) {};
        \coordinate (2323) at (4,-4) {};
      \end{scope}
      \begin{scope}[colouredcrystaledges]
        \draw[f2] (1212) to (1213);
        \draw[f2] (1213) to (1313);
        \draw[f3] (1213) to (1214);
        \draw[f4] (1214) to (1215);
        \draw[f2] (1214) to (1314);
        \draw[f3] (1313) to (1314);
        \draw[f1] (1313) to (1323);
        \draw[f5] (1215) to (1216);
        \draw[f2] (1215) to (1315);
        \draw[f4] (1314) to (1315);
        \draw[f1] (1314) to (1324);
        \draw[f3] (1314) to (1414);
        \draw[f1] (1323) to (2323);
        \draw[f3] (1323) to (1324);
      \end{scope}
      \fill[color=white,path fading=fade down] ($ (1215) + (-.2,-.5) $) rectangle ($ (1323) + (2,-.8) $);
      \fill[color=white] ($ (1215) + (-.2,-.79) $) rectangle ($ (1323) + (2,-1.1) $);
      \fill[color=white,path fading=fade right] ($ (1323) + (.35,0) $) rectangle ($ (1323) + (.6,-1) $);
      \fill[color=white] ($ (1323) + (.58,0) $) rectangle ($ (1323) + (2,-1) $);
    \end{scope}
    \begin{scope}
      \node[font=\scriptsize] (keyhead) at ($ (1251) + (4mm,-42mm) $) {\textit{Legend}};
      \node (keyf1l) at ($ (keyhead) + (10mm,-.5mm) $) {};
      \node (keyf1r) at ($ (keyf1l) + (10mm,0) $) {};
      \node (keyf2l) at ($ (keyf1l) + (15mm,0mm) $) {};
      \node (keyf2r) at ($ (keyf1r) + (15mm,0mm) $) {};
      \node (keyf3l) at ($ (keyf2l) + (15mm,0mm) $) {};
      \node (keyf3r) at ($ (keyf2r) + (15mm,0mm) $) {};
      \node (keyf4l) at ($ (keyf3l) + (15mm,0mm) $) {};
      \node (keyf4r) at ($ (keyf3r) + (15mm,0mm) $) {};
      \node (keyf5l) at ($ (keyf4l) + (15mm,0mm) $) {};
      \node (keyf5r) at ($ (keyf4r) + (15mm,0mm) $) {};
      \draw[lightgray] ($ (keyhead) + (-7mm,4mm) $) rectangle ($ (keyf5r) + (3mm,-3mm) $);
      \begin{scope}[smallcrystal,labelledcolouredcrystaledges]
        \draw[f1] (keyf1l) to (keyf1r);
        \draw[f2] (keyf2l) to (keyf2r);
        \draw[f3] (keyf3l) to (keyf3r);
        \draw[f4] (keyf4l) to (keyf4r);
        \draw[f5] (keyf5l) to (keyf5r);
      \end{scope}
    \end{scope}
    \node at ($ (1251) + (15mm,-17.5mm) $) {$\underbrace{\hbox{\vrule width 40mm height 0cm depth 0cm}}$};
    \node at ($ (2151) + (15mm,-17.5mm) $) {$\underbrace{\hbox{\vrule width 40mm height 0cm depth 0cm}}$};
    \node at ($ (1215) + (15mm,-17.5mm) $) {$\underbrace{\hbox{\vrule width 40mm height 0cm depth 0cm}}$};
    \node at ($ (1251) + (15mm,-28mm) $) {$\tikz[tinybst]
          \node {$4$}
          child { node {$1$} }
          child { node {$3$}
            child { node {$2$} }
            child[missing]
          };$};
    \node at ($ (2151) + (15mm,-28mm) $) {$\tikz[tinybst]
          \node {$4$}
          child { node {$2$} }
          child { node {$3$}
            child { node {$1$} }
            child[missing]
          };$};
    \node at ($ (1215) + (15mm,-28mm) $) {$\tikz[tinybst]
          \node {$4$}
          child { node {$3$}
            child { node {$1$} }
            child { node {$2$} }
          }
          child[missing];$};
    \node at ($ (1212) + (15mm,7.5mm) $) {$\overbrace{\hbox{\vrule width 40mm height 0cm depth 0cm}}$};
    \node at ($ (2121) + (15mm,7.5mm) $) {$\overbrace{\hbox{\vrule width 40mm height 0cm depth 0cm}}$};
    \node at ($ (1221) + (15mm,7.5mm) $) {$\overbrace{\hbox{\vrule width 40mm height 0cm depth 0cm}}$};
    \node at ($ (1212) + (15mm,12.5mm) $) {$\Gamma(\sylv,1212)$};
    \node at ($ (2121) + (15mm,12.5mm) $) {$\Gamma(\sylv,2121)$};
    \node at ($ (1221) + (15mm,12.5mm) $) {$\Gamma(\sylv,1221)$};
  \end{tikzpicture}
  \caption{Part of the connected components $\Gamma(\sylv,1221)$ and $\Gamma(\sylv,2121)$, which are linked by a
    shape-preserving isomorphism, and part of the connected component $\Gamma(\sylv,1212)$, which is only isomorphic to
    the others as a labelled digraph.}
  \label{fig:gammasylvthree}
\end{figure}

\begin{proposition}
  The relations $\sylvisom$ and $\sylvcong$ coincide.
\end{proposition}

\begin{proof}
  Let $u,v \in \aA^*$. Suppose that $u \sylvisom v$. Then $u \hypoisom v$ and hence $\wt{u} = \wt{v}$. Furthermore,
  $\sylvshape{u} = \sylvshape{v}$, and so
  \[
\Sh[\big]{\dectree{\std{u}^{-1}}} = \sylvshape{u} = \sylvshape{v} = \Sh[\big]{\dectree{\std{v}^{-1}}}.
  \]
  Hence, by \fullref{Corollary}{corol:charsylvcong}, $u \sylvcong v$.

  Now suppose that $u \sylvcong v$. Then $u \hypocong v$ and so $u \hypoisom v$ and so there is a quasi-crystal
  isomorphism $\theta : \Gamma(\hypo,u) \to \Gamma(\hypo,v)$ with $\theta(u) = v$. By \fullref{Corollary}{corol:charsylvcong},
  \[
    \sylvshape{u} = \Sh[\big]{\dectree{\std{u}^{-1}}} = \Sh[\big]{\dectree{\std{v}^{-1}}} = \sylvshape{v}.
  \]
  Thus, by S2, $\sylvshape{\theta(w)} = \sylvshape{w}$ for all $w \in \Gamma(\hypo,u)$. Hence $\theta$ is a
  $\sylvshapelit$-preserving quasi-crystal isomorphism and so $u \sylvisom v$.
\end{proof}

Our aim here has been to show that $\sylvcong$ can be defined using quasi-crystals; thus we have avoided using the prior
knowledge that $\sylvcong$ is a congruence \cite{hivert_sylvester}. If we had used that result, it would have been
straightforward to recover the fact that $\sylvshapelit$ is an abstract shape.

\subsection{Characterizing highest-weight elements}

\begin{proposition}
  \label{prop:lrpsameshapecomponent}
  Let $T$ be an unlabelled binary tree. The set of words in $\aA^*$ that are postfix readings of right strict binary
  search trees of shape $T$ form a single connected component of $\Gamma(\sylv)$.
\end{proposition}

\begin{proof}
  Suppose $u$ and $v$ are postfix readings of two right strict binary search trees of shape $T$. Then $\std{u}$ and $\std{v}$
  are also postfix readings of two binary search trees of shape $T$ \cite[Lemma~11]{hivert_sylvester}. Since there is exactly
  one right strict binary search of a given shape labelled by a standard word \cite[Note~3]{hivert_sylvester}, it
  follows that $\std{u} = \std{v}$ and so $u$ and $v$ are in the same connected component of $\Gamma(\sylv)$ by
  \fullref{Corollary}{corol:stdinsamecomp}.
\end{proof}

\begin{lemma}
  \label{lem:sylvchareidefined}
  Let $u \in \aA^*$ and $i \in \nset$. Then $\e_i(u)$ is defined if and only if $\rtree{u}$ contains at least one node
  $i+1$, but does not contain a node $i+1$ in the right subtree of the topmost node $i$.
\end{lemma}

\begin{proof}
  By definition, $\e_i(u)$ is defined if and only if every symbol $i+1$ is to the right of every symbol $i$ in $u$. In
  terms of \fullref{Algorithm}{alg:rightstrictinsert}, this is equivalent to every symbol $i+1$ being inserted
  \emph{before} every symbol $i$, which in turn is equivalent to no symbol $i+1$ being inserted into the right subtree
  of the topmost symbol $i$.
\end{proof}

\begin{proposition}
  \label{prop:sylvcharhighestweight}
  Let $u \in \aA^*$. Then $u$ is a highest-weight word if and only if $\rtree{u}$ has the following property: if the
  right interval partition of $\rtree{u}$ has $\ell$ parts, then for $a \in \set{1,\ldots,\ell} \subseteq \aA$, the nodes in
  the $a$-th right interval are labelled by $a$.
\end{proposition}

Before proceeding to the proof, an example is helpful. The word $u = 1121335432$ is highest weight, and $\rtree{u}$ is
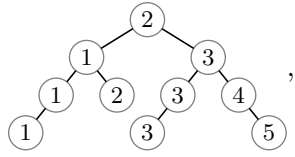
\[
\begin{tikzpicture}[baseline=-7.5mm]
  \begin{scope}[tinybst]
    \node (root) {$2$}
    child[sibling distance=16mm] { node (0) {$1$}
      child { node (00) {$1$}
        child { node (000) {$1$} }
        child[missing]
      }
      child { node (01) {$2$} }
    }
    child[sibling distance=16mm] { node (1) {$3$}
      child { node (10) {$3$}
        child { node (100) {$3$} }
        child[missing]
      }
      child { node (11) {$4$}
        child[missing]
        child { node (111) {$5$} }
      }
    };
  \end{scope}
\end{tikzpicture},
\]
which satisfies the condition in \fullref{Proposition}{prop:sylvcharhighestweight}. On the other hand, the word
$u = 5451761524$ does not have highest weight, and $\rtree{u}$ is the right-hand tree in \eqref{eq:bsteg}, which does
not satisfy the condition.

\begin{proof}
  Suppose $u$ is a highest-weight word. Then $\e_i(u)$ is undefined for all $i$. Let $\ell$ be the number of right
  intervals of $\rtree{u}$. Since the infix reading of $\rtree{u}$ is a weakly increasing word that must increase
  immediately after the rightmost nodes in the first $\ell-1$ right intervals, there are at least $\ell$ distinct symbols
  among the labels in $\rtree{u}$.

  If $i > 1$ is a symbol in $u$, then $i-1$ is also a symbol in $u$, for otherwise $\e_{i-1}(u)$ would be
  defined. Consequently, $u$ contains the symbols $1,\ldots,m$ for some $m \geq \ell$. If $m > 1$, then, by
  \fullref{Lemma}{lem:sylvchareidefined}, there is a symbol $2$ in the right subtree of the topmost symbol $1$. Thus the
  symbols $1$ must label all nodes in the first right interval of $\rtree{u}$. Proceeding inductively and using
  \fullref{Lemma}{lem:sylvchareidefined}, one sees that $a$ must label all nodes in the $a$-th right interval. Hence
  $\rtree{u}$ has the form described in the statement.

  Now suppose that $\rtree{u}$ has the form described. It is immediate that $a+1$ is in the right subtree of the topmost
  node $a$ for all $a \in \set{1,\ldots,\ell}$, and so $\e_i(u)$ is undefined for $i \leq \ell$. Since $\ell+1$ is the maximum
  symbol in $\rtree{u}$, it follows that $\e_i(u)$ is undefined for all $i$ by \fullref{Lemma}{lem:sylvchareidefined},
  and so $u$ has highest weight.
\end{proof}

Note that it would be possible to prove \fullref{Proposition}{prop:sylvcharhighestweight} by appealing to the
corresponding result for the hypoplactic monoid and quasi-ribbon tableaux \cite[Proposition~6(2)]{cm_hypoplactic} to
characterize highest weight words, but it seems clearer to give a direct proof that reasons only about trees.

\subsection{The sylvester version of the Robinson--Schensted correspondence}

The following result clarifies the interaction of the quasi-crystal structure and the sylvester version of the
Robinson--Schensted correspondence:

\begin{proposition}
  \label{prop:qeqiffsamecomponent}
  Let $u,v \in \aA^*$. The words $u$ and $v$ lie in the same connected component of $\Gamma(\sylv)$ if and only if
  $\qsylv{u} = \qsylv{v}$.
\end{proposition}

\begin{proof}
  Suppose $\qsylv{u} = \qsylv{v}$. Hence $\rrectree{u} = \rrectree{v}$ and by \fullref{Proposition}{prop:charrectrees},
  $\dectree{\std{u}^{-1}} = \dectree{\std{u}^{-1}}$. Thus, by \fullref{Proposition}{prop:infixreading}, $\std{u}^{-1}$
  and $\std{v}^{-1}$ are the infix readings of the same labelled binary tree, and so $\std{u} = \std{v}$. Since $u$ and
  $\std{u}$ lie in the same connected component of $\Gamma(\sylv)$ by \fullref{Corollary}{corol:stdinsamecomp}, and
  similarly for $v$ and $\std{v}$, it follows that $u$ and $v$ are in the same connected component of $\Gamma(\sylv)$.

  On the other hand, if $u$ and $v$ lie in the same connected component of $\Gamma(\sylv)$, then $\std{u} = \std{v}$ by
  \fullref{Corollary}{corol:stdinsamecomp} and so $\qsylv{u} = \qsylv{v}$.
\end{proof}

Thus the sylvester $\qlit$-symbol indexes connected components of the quasi-crystal graph and the sylvester
$\plit$-symbol locates a word within that component. Note that combining \fullref{Proposition}{prop:qeqiffsamecomponent}
and \fullref{Lemma}{lem:sylvcanonicalrobinsonschensted}, one recovers \fullref{Proposition}{prop:lrpsameshapecomponent}.

The Knuth `hook-length' formula shows that the number of descending trees of the same shape as $T$ is
$n!/\prod_{v \in T} h_v$, where $h_v$ is the number of nodes in the subtree rooted at $v$ \cite[\S~5.1.4,
Exer.~20]{knuth_taocp3}. This gives a formula for the number of sylvester $\qlit$-symbols of a given shape, and thus for
the size of sylvester classes corresponding to left strict binary search trees of a given shape, and thus for the number
of quasi-crystal components that are mapped to a given one by a $\sylvshapelit$-preserving quasi-crystal isomorphism.

In the infinite-rank setting, there are infinitely many sylvester $\plit$-symbols of a given shape. On restricting to
the rank $n$ case, the number of sylvester $\plit$-symbols of a given shape can be calculated using the same formula as
for the hypoplactic monoid \cite[Theorem~4]{cm_hypoplactic}:

\begin{proposition}
  Let $T$ be a binary search tree for which the right interval partition has $\ell$ parts. The number of right strict binary
  search trees of shape $T$ labelled by symbols from $\aA_n$ is
  \[
  \begin{cases}
    \binom{n+\tlen{T}-\ell}{n-\ell} & \text{if $\ell \leq n$} \\
    0 & \text{if $\ell > n$}.
  \end{cases}
  \]
\end{proposition}

\begin{proof}
  Consider a highest-weight word $u \in \aA_n^*$ such that $\rtree{u}$ has shape $T$. By
  \fullref{Proposition}{prop:sylvcharhighestweight}, $u$ contains the symbols $\set{1,\ldots,\ell}$. Let $\alpha$ be the
  shape of the quasi-ribbon tableau $\phypo{u}$. Since $u$ is also a highest-weight word in $\Gamma(\sylv)$ and thus in
  $\Gamma(\hypo)$, it follows that the column reading $v$ of $\phypo{u}$ is also a highest weight word in $\Gamma(\hypo)$ and so
  $\alpha = \wt{v} = \wt{u}$ \cite[Proposition~6]{cm_hypoplactic}. In particular, $\clen{\alpha} = \ell$ and
  $\cwt{\alpha} = \wlen{v} = \wlen{u} = \tlen{T}$.

  Since $\Gamma(\sylv,u) = \Gamma(\hypo,u)$ is isomorphic to $\Gamma(\hypo,v)$, they both contain the same number of
  words that lie in $\aA_n^*$, and so the result follows by subsituting the given values for $\clen{\alpha}$ and
  $\cwt{\alpha}$ into \cite[Theorem~4]{cm_hypoplactic}.
\end{proof}

\subsection{Counting factorizations}

One interpretation of the Littlewood--Richardson rule \cite[ch.~5]{fulton_young} is that the Littlewood--Richardson
coefficients $c_{\lambda\mu}^\nu$ are the number of different factorizations of an element of $\plac$ corresponding to a
Young tableau of shape $\nu$ into elements corresponding to a tableau of shape $\lambda$ and a tableau of shape
$\mu$. In particular, it shows that the number of such factorizations is independent of the content of the tableau, and
is rather dependent only on the shape. The quasi-crystal structure yields a similar result for $\hypo$
\cite[Theorem~5]{cm_hypoplactic}, and a minimal modification to the proof gives the similar result for the sylvester
monoid:

\begin{theorem}
  \label{thm:sylvfactorizations}
  The number of distinct factorizations of an element of the sylvester monoid corresponding to a right strict binary
  search tree of shape $T$ into elements that correspond to right strict binary search trees of shapes $U$ and
  $V$ is dependent only of $T$, $U$, and $V$, and not on the content of the element.
\end{theorem}

\begin{proof}
  For the purposes of this proof, a word $u \in \aA^*$ is a \defterm{postfix word} if it is the postfix reading of a
  binary search tree; a postfix word $u$ has shape $U$, where $U$ is an unlabelled binary tree, if it is the postfix
  reading of a binary search tree of shape $U$.

  Let $T$, $U$, and $V$ be unlabelled binary trees. Let $w \in \aA^*$ be a postfix word of shape $T$. Let
  \begin{align*}
  S^w_{U,V} = \gsetsplit[\big]{(u,v)}{\;&\text{$u,v \in \aA_n^*$ are postfix words,} \\
    &\text{$u$ has shape $U$,} \\
    &\text{$v$ has shape $V$,} \\
    &w \sylvcong uv}.
  \end{align*}
  So $S^w_{U,V}$ is a complete list of factorizations of $w$ into elements whose corresponding right strict
  binary search trees have shapes $U$ and $V$.  Let $i \in \nset$ and suppose $\e_i(w)$ is defined. Pick some
  pair $(u,v) \in S^w_{U,V}$. Then $w \sylvcong uv$ and so $\e_i(uv)$ is defined, and $\e_i(w) = \e_i(uv)$. By
  \cite[Lemma~2]{cm_hypoplactic}, either $\e_i(uv) = u\e_i(v)$, or $\e_i(uv) = \e_i(u)v$. In the former case,
  $(u,\e_i(v)) \in S^{\e_i(w)}_{U,V}$, and in the latter case $(\e_i(u),v) \in S^{\e_i(w)}_{U,V}$, since $\e_i$ preserves being a
  postfix word and the shape of the corresponding right strict binary search tree by
  \fullref{Proposition}{prop:lrpsameshapecomponent}. So $\e_i$ induces an injective map from $S^w_{U,V}$ to
  $S^{\e_i(w)}_{U,V}$. It follows that $\f_i$ induces the inverse map from $S^{\e_i(w)}_{U,V}$ to
  $S^w_{U,V}$. Hence $\abs{S^w_{U,V}} = \abs{S^{\e_i(w)}_{U,V}}$. Similarly, if $\f_i(w)$ is
  defined, then $\abs{S^w_{U,V}} = \abs{S^{\f_i(w)}_{U,V}}$. Since all the postfix words whose
  corresponding binary search trees have shape $T$ lie in the same connected component of $\Gamma(\sylv)$ by
  \fullref{Proposition}{prop:lrpsameshapecomponent}, it follows that $|S^w_{U,V}|$ is dependent only on
  $T$, not on $w$.
\end{proof}

\subsection{Satisfying an identity}

Another application of the quasi-crystal structure was to prove that the hypoplactic monoid satisfies a non-trivial
identity \cite[Theorem~6]{cm_hypoplactic}. Similarly, it is possible to use the quasi-crystal structure to prove that
the sylvester monoid satisfies an identity.

\begin{theorem}
  \label{thm:sylvidentity}
  The sylvester monoid satisfies the identity $xyxy = yxxy$.
\end{theorem}

\begin{proof}
  Let $x,y \in \aA^*$. The proof that $xyxy = yxxy$ proceeds by reverse induction on the weight of $xyxy$.

  The base case of the induction is when $xyxy$ is highest-weight. Thus $\e_i(xyxy)$ is undefined for all $i \in \nset$.
  So $xyxy$ must contain a subsequence $(i+1)i$ for all $i \in \set{1,\ldots,\ell-1}$, where $\ell$ is the maximum
  symbol appearing in $x$ or $y$.  The symbols $i+1$ and $i$ may each lie in $x$ or $y$, but in any case, there is a
  subsequence $(i+1)i$ in $yxxy$. Hence $\e_i(yxxy)$ is undefined for all $i \in \nset$ and so $yxyx$ is also
  highest-weight. Clearly $\wt{xyxy} = \wt{yxxy}$. Consider applying \fullref{Algorithm}{alg:rightstrictinsert} to
  $xyxy$ and $yxxy$. In both cases, the algorithm initially computes $\rtree{xy}$. At this point, every symbol in
  $B = \set{1,\ldots,\ell}$ appears in the tree and so the infix reading of $\rtree{xy}$ contains all these
  symbols. Thus any symbol $b \in B$ that is subsequently inserted will be placed either as the left child of a node
  $b$ that is already present, or into the right subtree of the topmost node $b-1$ that is already present. Thus the
  trees $\rtree{xyxy}$ and $\rtree{yxxy}$ are identical.

  The remaining reasoning parallels the proof of \cite[Theorem~6]{cm_hypoplactic}. For the induction step, suppose
  $xyxy$ is not highest-weight, and that $x'y'x'y' \sylvcong y'x'x'y'$ for all $x',y' \in \aA^*$ such that $x'y'x'y'$
  has higher weight than $xyxy$. Then $\e_i(xyxy)$ is defined for some $i \in \nset$. Neither $x$ nor $y$ contains a
  symbol $i$, since otherwise there would be a subsequence $(i+1)i$. So
  $\ecount_i(xyxy) = \ecount_i(yxxy) = \abs{xyxy}_{i+1} = \abs{yxxy}_{i+1} = 2\abs{x}_{i+1} + 2\abs{y}_{i+1} =
  2\ecount_i(x) + 2\ecount_i(y)$,
  and $\ecount_i(xyxy)$ applications of $\e_i$ change every symbol $i+1$ in $xyxy$ to $i$. That is,
  $\e_i^{\ecount_i(xyxy)}(xyxy)$ and $\e_i^{\ecount_i(yxxy)}(yxxy)$ are both defined and are equal to $x'y'x'y'$ and
  $y'x'x'y'$ respectively, where $x' = \e_i^{\ecount_i(x)}(x)$ and $y' = \e_i^{\ecount_i(y)}(y)$. Since $x'y'x'y'$ has
  higher weight than $xyxy$, it follows by the induction hypothesis that $x'y'x'y' \sylvcong y'x'x'y'$. Hence
  \[
  xyxy = \f_i^{\ecount_i(xyxy)}(x'y'x'y') \sylvcong \f_i^{\ecount_i(xyxy)}(y'x'x'y') = yxxy.
  \]
  This completes the induction step and thus the proof.
\end{proof}

For an alternative proof of \fullref{Theorem}{thm:sylvidentity} that does not require the quasi-crystal structure, see
\cite[Section~3.3]{cm_identities}.

\section{Crystallizing the Baxter monoid}
\label{sec:baxtcrystal}

\subsection{Definition by crystals}

Let $S$ be the set of pairs of twin binary trees and define
$\baxtshape{u} = \parens[\big]{\Sh{\inctree{\std{u}^{-1}}},\Sh{\dectree{\std{u}^{-1}}}}$ for all $u \in \aA^*$.

\begin{proposition}
  \label{prop:baxtshape}
  $\baxtshapelit$ is an abstract shape.
\end{proposition}

\begin{proof}
  By \fullref{Proposition}{prop:eipreservesstd}, the map $\baxtshapelit$ is invariant under the quasi-Kashiwara operators,
  and so satisfies axiom~S1.

  To see that $\baxtshapelit$ satisfies~S2, proceed as follows. Let $u,v \in \aA^*$ be such that
  $\wt{u} = \wt{v}$ and $\baxtshape{u} = \baxtshape{v}$, and let $a \in \aA$.

  Then $\Sh{\inctree{\std{u}^{-1}}} = \Sh{\inctree{\std{v}^{-1}}}$ and
  $\Sh{\dectree{\std{u}^{-1}}} = \Sh{\dectree{\std{v}^{-1}}}$. By the reasoning in the proof of
  \fullref{Proposition}{prop:sylvshape}, $\Sh{\dectree{\std{ua}^{-1}}} = \Sh{\dectree{\std{va}^{-1}}}$ and
  $\Sh{\dectree{\std{au}^{-1}}} = \Sh{\dectree{\std{av}^{-1}}}$. Symmetrical reasoning, but using part~(2) of
  \fullref{Lemmata}{lem:decinctreenewmaxshape} and \ref{lem:decinctreenewminshape} shows that
  $\Sh{\inctree{\std{ua}^{-1}}} = \Sh{\inctree{\std{va}^{-1}}}$ and
  $\Sh{\inctree{\std{au}^{-1}}} = \Sh{\inctree{\std{av}^{-1}}}$. Hence $\baxtshape{ua} = \baxtshape{va}$ and
  $\baxtshape{au} = \baxtshape{av}$. Hence $\baxtshapelit$ satisfies~S2.
\end{proof}

As a consequence of \fullref{Proposition}{prop:baxtshape}, one can define a relation $\baxtisom$ on the free monoid
$\aA^*$ to be the relation $\shapisom$ from \fullref{Section}{sec:abstractshapes} with the abstract shape
$\shapelit = \baxtshapelit$. By \fullref{Propositions}{prop:isomdefinescong} and \ref{prop:baxtshape}, $\baxtisom$ is a
congruence.

When viewing the quasi-crystal graph $\Gamma(\hypo)$ in the context of $\baxtshapelit$-preserving isomorphisms, denote it
by $\Gamma(\baxt)$. Similarly, denote the connected component $\Gamma(\hypo,u)$ by $\Gamma(\baxt,u)$. (This is arguably
an abuse of notation, since $\Gamma(\hypo)$ and $\Gamma(\baxt)$ are the same graph, and the difference is in the notion
of isomorphism and the relation to which it gives rise.)

\begin{proposition}
  The relations $\baxtisom$ and $\baxtcong$ coincide.
\end{proposition}

\begin{proof}
  Let $u,v \in \aA^*$. Suppose that $u \baxtisom v$. Then $u \hypoisom v$ and hence $\wt{u} = \wt{v}$. Furthermore,
  $\baxtshape{u} = \baxtshape{v}$, and so
  \begin{multline*}
    \parens[\big]{\Sh{\inctree{\std{u}^{-1}}},\Sh{\dectree{\std{u}^{-1}}}} = \baxtshape{u} \\
    = \baxtshape{v} = \parens[\big]{\Sh{\inctree{\std{v}^{-1}}},\Sh{\dectree{\std{v}^{-1}}}}
  \end{multline*}
  Hence, by \fullref{Corollary}{corol:charbaxtcong}, $u \baxtcong v$.

  Now suppose that $u \baxtcong v$. Then $u \hypocong v$ and so $u \hypoisom v$ and so there is a quasi-crystal
  isomorphism $\theta : \Gamma(\hypo,u) \to \Gamma(\hypo,v)$. By \fullref{Corollary}{corol:charbaxtcong},
  \begin{multline*}
    \baxtshape{u} = \parens[\big]{\Sh{\inctree{\std{u}^{-1}}},\Sh{\dectree{\std{u}^{-1}}}} \\
    = \parens[\big]{\Sh{\inctree{\std{v}^{-1}}},\Sh{\dectree{\std{v}^{-1}}}} = \baxtshape{v}.
  \end{multline*}
  Thus, by S2, $\baxtshape{\theta(w)} = \baxtshape{w}$ for all $w \in \Gamma(\hypo,u)$. Hence $\theta$ is a
  $\baxtshapelit$-preserving quasi-crystal isomorphism and so $u \baxtisom v$.
\end{proof}

As was the case for $\sylvcong$, our aim here has been to show that $\baxtcong$ can be defined using quasi-crystals;
thus we have avoided using the prior knowledge that $\baxtcong$ is a congruence \cite{giraudo_baxter2}, from which, it
would have been straightforward to recover the fact that $\baxtshapelit$ is an abstract shape.

\subsection{Characterizing highest weight elements}

\begin{proposition}
  \label{prop:leftconsistentsameshapecomponent}
  Let $(T_L,T_R)$ be a pair of twin binary trees. The set of words in $\aA^*$ that are left-consistent readings of pairs
  of twin binary search trees of shape $(T_L,T_R)$ form a single connected component of $\Gamma(\baxt)$.
\end{proposition}

\begin{proof}
  Suppose $u$ and $v$ are left-consistent readings of two pairs of twin binary seach trees of shape $(T_L,T_R)$. Then
  $\std{u}$ and $\std{v}$ are also left-consistent readings of shape $(T_L,T_R)$ (by \cite[Lemma~11]{hivert_sylvester}
  and its dual). Since there is exactly one pair of twin binary search trees of a given shape labelled by a standard
  word (by \cite[Note~3]{hivert_sylvester} and its dual), it follows that $\std{u} = \std{v}$ and so $u$ and $v$ are in
  the same connected component of $\Gamma(\baxt)$ by \fullref{Corollary}{corol:stdinsamecomp}.
\end{proof}

The following lemma is the dual of \fullref{Lemma}{lem:sylvchareidefined}, and can be proved by a dual argument.

\begin{lemma}
  \label{lem:dualsylvchareidefined}
  Let $u \in \aA^*$ and $i \in \nset$. Then $\e_i(u)$ is defined if and only if $\ltree{u}$ contains at least one node
  $i+1$, but does not contain a node $i$ in the left subtree of the topmost node $i+1$.
\end{lemma}

The following result follows from the argument in the proof of \fullref{Proposition}{prop:sylvcharhighestweight} plus a dual
argument:

\begin{proposition}
  \label{prop:baxtcharhighestweight}
  Let $u \in \aA^*$. Then $u$ is a highest-weight word if and only if $\pbaxt{u} = \parens{\ltree{u},\rtree{u}}$ has the
  following properties:
  \begin{itemize}
    \item if the left interval partition of $\ltree{u}$ has $k$ parts, then for
    $a \in \set{1,\ldots,k} \subseteq \aA$, the nodes in the $a$-th left interval are labelled by $a$;
    \item if the right interval partition of $\rtree{u}$ has $\ell$ parts, then for
    $a \in \set{1,\ldots,\ell} \subseteq \aA$, the nodes in the $a$-th right interval are labelled by $a$.
  \end{itemize}
\end{proposition}

\subsection{The Baxter version of the Robinson--Schensted correspondence}

The following result is the parallel of \fullref{Proposition}{prop:qeqiffsamecomponent} for the Baxter monoid:

\begin{proposition}
  \label{prop:qbaxteqiffsamecomponent}
  Let $u,v \in \aA^*$. The words $u$ and $v$ lie in the same connected component of $\Gamma(\baxt)$ if and only if
  $\qbaxt{u} = \qbaxt{v}$.
\end{proposition}

Like the case of the sylvester monoid, the Baxter $\qlit$-symbol indexes connected components of the quasi-crystal graph
and the Baxter $\plit$-symbol locates a word within that component.

Unlike the sylvester monoid, there is no known neat formula for the size of classes of words that represent the same element
of the Baxter monoid. However, the quasi-crystal structure does make it clear that the size of these classes is
dependent only on the shape of the corresponding pair of twin binary search trees.

\begin{remark}
  One might consider trying to use the quasi-crystal structure to obtain a formula for the size of these classes as
  follows. By the quasi-crystal structure, it suffices to consider the size of the class for a standard element;
  that is, for a pair of twin binary search trees $(T_L,T_R)$ containing each symbol in $K = \set{1,\ldots,k}$ exactly
  once. Each of these binary search trees can be considered as a partial order on the set $K$; denote these partial orders
  by $\leq_L$ and $\leq_R$. Let $\leq_L'$ be the dual of $\leq_L$. Now, any reading of $(T_L,T_R)$ can be viewed as a
  linear order on $K$ that extends $\leq'_L$ and $\leq_R$. Since $K$ is finite, there are finitely many such linear
  orders, so we can intersect them to obtain a partial order $\sqsubset$. So the readings of $(T_L,T_R)$ are in
  one-to-one correspondence with the linear orders extending $\sqsubset$. However, the problem of the number of linear
  orders extending a given partial order is $\#P$-complete \cite{brightwell_counting}.

  Of course, not all partial orders on $K$ will arise as $\sqsubset$, so it is possible that a more refined version of
  this argument, using the structure of the pair of twin binary search trees to gain information about $\sqsubset$,
  might still work.
\end{remark}

\subsection{Counting factorizations}

The following result is the parallel of \fullref{Theorem}{thm:sylvfactorizations} for the Baxter monoid:

\begin{theorem}
  The number of distinct factorizations of an element of the Baxter monoid corresponding to a pair of twin binary search
  trees of shape $(T_L,T_R)$ into elements that correspond to pairs of twin binary search trees of shapes $(U_L,U_R)$ and $(V_L,V_R)$ is
  dependent only of $(T_L,T_R)$, $(U_L,U_R)$, and $(V_L,V_R)$, and not on the content of the element.
\end{theorem}

\begin{proof}
  For the purposes of this proof, a word $u \in \aA^*$ is a \defterm{left-consistent word} if it is the left-consistent
  reading of a pair of twin binary search tree; a left-consistent word $u$ has shape $(U_L,U_R)$, where $(U_L,U_R)$ is a
  pair of twin binary trees, if it is the left-consistent reading of a pair of twin binary search trees of shape
  $(U_L,U_R)$.

  Let $(T_L,T_R)$, $(U_L,U_R)$, and $(V_L,V_R)$ be pairs of twin binary trees. Let $w \in \aA^*$ be a postfix word such that
  $T$ is the shape of $\rtree{w}$. Let
  \begin{align*}
  S^w_{(U_L,U_R),(V_L,V_R)} = \gsetsplit[\big]{(u,v)}{\;&\text{$u,v \in \aA_n^*$ are left-consistent words,} \\
    &\text{$u$ has shape $(U_L,U_R)$,} \\
    &\text{$v$ has shape $(V_L,V_R)$,} \\
    &w \baxtcong uv}.
  \end{align*}
  Now following the reasoning in the proof of \fullref{Theorem}{thm:sylvfactorizations}, but using
  \fullref{Propositions}{prop:qbaxteqiffsamecomponent} and \ref{prop:leftconsistentsameshapecomponent} shows that
  $|S^w_{(U_L,U_R),(V_L,V_R)}|$ is dependent only on $(T_L,T_R)$, not on $w$.
\end{proof}

\subsection{Satisfying an identity}

The analogy of \fullref{Theorem}{thm:sylvidentity} for the Baxter monoid is the following:

\begin{theorem}
  \label{thm:baxtidentity}
  The Baxter monoid satisfies the identity $xyxyxy = xyyxxy$.
\end{theorem}

\begin{proof}
  Let $x,y \in \aA^*$. The proof that $xyxyxy = xyyxxy$ proceeds by reverse induction on the weight of $xyxyxy$.

  The base case of the induction is when $xyxyxy$ is highest-weight. By reasoning parallel to the proof of
  \fullref{Theorem}{thm:sylvidentity}, $xyyxxy$ is also highest-weight. Clearly $\wt{xyxyxy} = \wt{xyyxxy}$. Again by
  reasoning parallel to the proof of \fullref{Theorem}{thm:sylvidentity}, the trees $\rtree{xyxyxy}$ and $\rtree{xyyxxy}$
  are equal. By dual reasoning, the trees $\rtree{xyxyxy}$ and $\rtree{xyyxxy}$ are equal.

  The induction step is essentially identical to the proof of \fullref{Theorem}{thm:sylvidentity}, except that
  $\ecount_i(xyxyxy) = 3\ecount_i(x) + 3\ecount_i(y) = \ecount_i(xyyxxy)$ applications of $\e_i$ change every symbol
  $i+1$ to $i$. The result follows.
\end{proof}

For an alternative proof of \fullref{Theorem}{thm:baxtidentity} that does not require the quasi-crystal structure, see
\cite[Section~3.5]{cm_identities}.

\section{Perspectives}

\subsection{The Robinson--Schensted type correspondences}

The graphs $\Gamma(\hypo)$, $\Gamma(\sylv)$ and $\Gamma(\baxt)$ are all the same object; the difference is in the notion
of `isomorphic' components. In terms of the hypoplactic, sylvester, and Baxter versions of the Robinson--Schensted
correspondence, the $\qlit$-symbol of a word identifies a connected component of the graph, and the $\plit$-symbol of a
word identifies its position within that connected component. Viewed in this context, the sylvester and Baxter
$\qlit$-symbols are overdetermined. The `underlying' $\qlit$-symbol of a word $u \in \aA^*$ is $\std{u}$, which lies in
the same component as $u$. This corresponds with the notions of `compatibility with the destandardization process' and
$\stdlit$-goodness developed by \cite{priez_lattice}.

\subsection{Bases of combinatorial Hopf algebras}

In the infinite-rank setting, there is a very natural connection between the crystal graph $\Gamma(\plac)$ and free
Schur functions. The free Schur functions are indexed by standard Young tableau and form a basis for the algebra of free
symmetric functions $\mathbf{FSym}$. The crystal graph offers an equivalent way of defining them:
\[
\mathbf{S}_T = \sum_{w \in \Gamma(\plac,T)} w,
\]
where $\Gamma(\plac,T)$ denotes the connected component of $\Gamma(\plac)$ corresponding to a standard Young tableau $T$.

Similarly, the free quasi-ribbon functions, which form a basis for $\mathbf{FQSym}$ can be defined with reference to the quasi-crystal graph:
\[
\mathbf{F}_u = \sum_{w \in \Gamma(\hypo,u)} w,
\]
where $u \in \aA^*$ is a standard word; this follows from the fact that standard words index connected components of
$\Gamma(\hypo)$.

For the basis of the algebra of planar binary trees $\mathbf{PBT}$ given by \cite[Eqs~(16--17)]{hivert_sylvester} (see also \cite[Theorem~27]{hivert_hopf}), the definition using the quasi-crystal graph becomes:
\[
\mathbf{P}_T = \sum_{\substack{w \in \Gamma(\sylv,U)\\\Sh{U} = T}} w,
\]
where $T$ is a binary tree, $U$ ranges over the set of decreasing trees and $\Gamma(\sylv,U)$ denotes the connected
component of $\Gamma(\sylv)$ indexed by the decreasing tree $U$.

For the basis of the algebra $\mathbf{Baxter}$ given in \cite[Theorem~6.12]{giraudo_baxter2}, the definition using the
quasi-crystal graph is:
\[
\mathbf{P'}_{(T_L,T_R)} = \sum_{\substack{w \in \Gamma(\baxt,(U_L,U_R))\\\Sh{U_L,U_R} = (T_L,T_R)}} w,
\]
where $(T_L,T_R)$ is a pair of twin binary trees, $(U_L,U_R)$ ranges over the set of pairs made up of increasing and a
decreasing tree with the same infix reading and $\Gamma(\baxt,(U_L,U_R))$ denotes the connected component of
$\Gamma(\baxt)$ indexed by the pair $(U_L,U_R)$.

\subsection{Speculations}

In the quasi-crystal structure for the hypoplactic monoid, one constructs the graph using the quasi-Kashiwara operators
and one defines the congruence using labelled directed graph isomorphisms of connected components. If one replaces the
quasi-Kashiwara operators with the original Kashiwara operators, one obtains the plactic monoid. On the other hand, if
one requires that labelled directed graph isomorphisms be shape-preserving, one obtains the sylvester and Baxter
monoids. It is thus natural to ask what monoids arise when one makes both replacements and considers shape-preserving
isomorphisms on the graph arising from the original Kashiwara operators (see \fullref{Figure}{fig:venn}). (If one uses
the abstract notion of shape defined in \fullref{Section}{sec:abstractshapes} and the original Kashiware operators,
axiom~S1 should probably be replaced with `invariance under the Kashiwara operators', which is a natural analogue in
light of the discussion following axioms~S1 and~S2. Alternatively, one could seek a natural analogue of standardization
for the original Kashiwara operators.)

\begin{figure}[t]
  \centering
  \begin{tikzpicture}[x=18mm,y=18mm,every node/.style={align=center}]
    \pgfdeclarelayer{boxes}
    \pgfdeclarelayer{nodes}
    \pgfsetlayers{boxes,nodes,main}
    \begin{pgfonlayer}{nodes}
      \node (hypo) at (0,-1) {$\hypo$};
      \node (plac) at (-1,0) {$\plac$};
      \node[align=center] (sylv) at (1,0) {$\sylv$\\$\baxt$};
      \node[font=\Large] (half) at (0,1) {?};
      %
      %
      \node[rotate=-45,anchor=south,font=\small] at ($ (sylv) + (.5,.5) $) {Quasi-Kashiwara\\operators};
      \node[rotate=-45,anchor=south,font=\small] at ($ (half) + (.5,.5) $) {Kashiwara\\operators};
      \node[rotate=45,anchor=south,font=\small] at ($ (half) + (-.5,.5) $) {Shape-preserving\\isomorphisms};
      \node[rotate=45,anchor=south,font=\small] at ($ (plac) + (-.5,.5) $) {Isomorphisms};
    \end{pgfonlayer}
    \begin{pgfonlayer}{boxes}
      \path[rounded corners,fill=lightgray,opacity=.5] ($ (half)  + (-.9,.1) $) -- ++(.8,.8) -- ($ (sylv) + (.9,-.1) $) -- ++ (-.8,-.8) -- cycle;
      \path[rounded corners,fill=lightgray,opacity=.5] ($ (half)  + (.9,.1) $) -- ++(-.8,.8) -- ($ (plac) + (-.9,-.1) $) -- ++ (.8,-.8) -- cycle;
      \path[rounded corners,fill=lightgray,opacity=.5] ($ (plac)  + (-.9,.1) $) -- ++(.8,.8) -- ($ (hypo) + (.9,-.1) $) -- ++ (-.8,-.8) -- cycle;
      \path[rounded corners,fill=lightgray,opacity=.5] ($ (sylv)  + (.9,.1) $) -- ++(-.8,.8) -- ($ (hypo) + (-.9,-.1) $) -- ++ (.8,-.8) -- cycle;
    \end{pgfonlayer}
  \end{tikzpicture}
  \caption{(Quasi-)-Kashiwara operators, notions of isomorphisms, and monoids.}
  \label{fig:venn}
\end{figure}
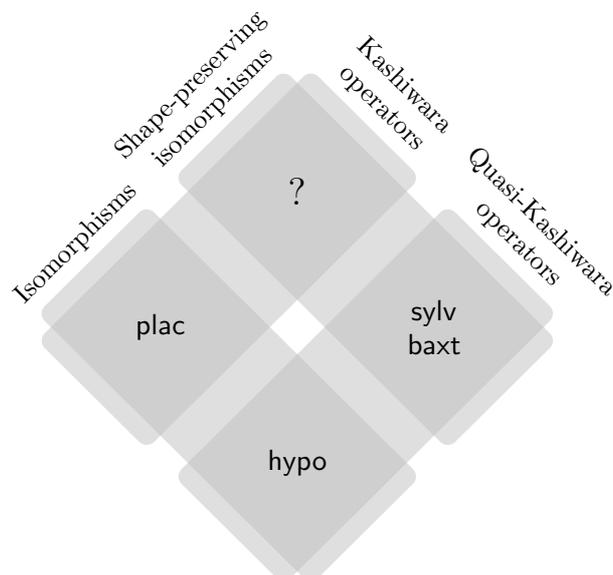

\bibliography{\jobname}
\bibliographystyle{alphaabbrv}

\end{document}